\RequirePackage{fix-cm}
\documentclass[AMS,Times1COL]{NJDv5}

\articletype{}

\received{Date Month Year}
\revised{Date Month Year}
\accepted{Date Month Year}
\journal{Numer Methods Partial Differential Eq.}
\volume{00}
\copyyear{2024}
\startpage{1}

\usepackage{scalefnt}

\usepackage[utf8]{inputenc}
\usepackage{graphicx,xcolor,float}
\usepackage{amsmath,amssymb,bm,pgfplots}
\usepackage[sort,compress]{cite}
\usepackage[capitalise,nameinlink]{cleveref}
\pgfplotsset{compat=1.18}
\renewcommand{\cite}{\citep}
\usepackage{mathtools,enumitem}
\usepackage[utf8]{inputenc}
\usepackage{grffile,tikz,comment}

\def\div{\operatorname{div}}

\def\CC{\mathbf{C}}
\def\E{\mathcal{E}}
\def\G{\mathcal{G}}
\def\T{\mathcal{T}}
\def\R{\mathbb{R}}

\usepackage{multirow}
\def\Th{\mathcal{T}_h}
\def\Itau{\mathcal{I}_\tau}
\def\I{\mathbf{I}}

\def\la{\langle}
\def\ra{\rangle}
\def\lla{\left\langle}
\def\rra{\right\rangle}
\def\Vh{\mathcal{V}_h}

\def\Xh{\mathcal{X}_h}
\newcommand\restr[2]{{% we make the whole thing an ordinary symbol
  \left.\kern-\nulldelimiterspace % automatically resize the bar with \right
  #1 % the function
  \vphantom{\big|} % pretend it's a little taller at normal size
  \right|_{#2} % this is the delimiter
  }}
  \newcommand\restrb[3]{{% we make the whole thing an ordinary symbol
  \left.\kern-\nulldelimiterspace % automatically resize the bar with \right
  #1 % the function
  \vphantom{\big|} % pretend it's a little taller at normal size
  \right|_{#2}^{#3} % this is the delimiter
  }}
\newcommand{\dx}{\, \mathrm{d}\mathbf{x}}

\def\div{\operatorname{div}}

\DeclarePairedDelimiter{\norm}{\|}{\|}
\DeclarePairedDelimiter{\snorm}{|}{|}

\def\dt{\partial_t}

\def\dtau{d^{n+1}_\tau}
\def\ds{\,\mathrm{d}s}
\def\u{\mathbf{u}}
\def\vv{\mathbf{v}}

\def\tr#1{{\mathrm{tr}(#1)}}
\def\ddt{\frac{\mathrm{d}}{\mathrm{d}t}}

\def\softd{{\leavevmode\setbox1=\hbox{d}%
		\hbox to 1.05\wd1{d\kern-0.4ex{\char039}\hss}}}%cstocs

\begin{document}
\scalefont{1.3}

\title{\huge Structure-preserving approximation of the Cahn-Hilliard-Biot system}
\author[1]{Aaron Brunk}
\author[2]{Marvin Fritz}
\authormark{BRUNK \textsc{and} FRITZ}
\titlemark{Structure-preserving approximation of the Cahn-Hilliard-Biot system}
\address[1]{\orgdiv{Institute of Mathematics}, \orgname{Johannes Gutenberg University}, \orgaddress{\state{Mainz}, \country{Germany}}}
\address[2]{\orgdiv{Computational Methods for PDEs}, \orgname{Johann Radon Institute for Computational and Applied Mathematics}, \orgaddress{\state{Linz}, \country{Austria}}}
\def\myalpha{\alpha}
\corres{Aaron Brunk, Institute of Mathematics, Johannes Gutenberg University , Staudinger Weg 9, 55099  Mainz, Germany. Email: abrunk@uni-mainz.de}

\abstract[Abstract]{\large In this work, we propose a structure-preserving discretisation for the recently studied Cahn-Hilliard-Biot system using conforming finite elements in space and problem-adapted explicit-implicit Euler time integration. We prove that the scheme preserves the thermodynamic structure, that is, the balance of mass and volumetric fluid content and the energy dissipation balance. The existence of discrete solutions is established under suitable growth conditions. Furthermore, it is shown that the algorithm can be realised as a splitting method, that is, decoupling the Cahn-Hilliard subsystem from the poro-elasticity subsystem, while the first one is nonlinear and the second subsystem is linear. The schemes are illustrated by numerical examples and a convergence test.}


\jnlcitation{\cname{%
\author{Brunk A.}, and
\author{Fritz M}}.
\ctitle{Structure-preserving approximation of the Cahn-Hilliard-Biot system.} \cjournal{\it Numer Methods Partial Differential Eq.} \cvol{2024;00(00):00--00}.}
\maketitle
\counterwithin{lemma}{section}
\counterwithin{equation}{section}

\section{Introduction} 
In this paper, we focus on the development and analysis of a structure-preserving discretisation scheme for the Cahn-Hilliard-Biot system. This model, originally derived in \cite{STORVIK2022}, combines the Cahn-Hilliard equation, which describes phase separation and interfacial dynamics, with mechanical deformation and the Biot equations, which govern fluid flow. The system reads as follows:
	\begin{align}
		\dt \phi  &- \div\left( m(\phi ) \nabla \mu \right) = r(\phi, \E(\u), \theta),\\
		&\begin{aligned}[b]
			\! \! \! \mu = -\gamma \Delta \phi + \Psi'(\phi) + W_{\phi} (\phi, \E(\u)) &-M(\phi) (\theta - \alpha(\phi) \div(\u)) \alpha'(\phi) \div(\u)  \\ 
			& +  \frac{M'(\phi)}{2} (\theta - \alpha(\phi)\div(\u))^2,
		\end{aligned}  \label{eq:sys2} \\ 
		 \div(\bm{\sigma}) &= \mathbf{f} , \label{eq:sys3} \\ 
		\bm{\sigma} &= W_{\E} (\phi, \E(\u))  +\CC_{\nu}(\phi)  \E( \dt\u)  - \alpha (\phi) M(\phi) (\theta - \alpha(\phi) \div(\u))\I, \\ 
		\dt\theta &-  \div(\kappa(\phi) \nabla p) = s(\phi, \E(\u), \theta) 	, \label{eq:sys4}\\   
		p& = M(\phi) (\theta - \alpha (\phi) \div(\u)).
	\end{align}
The phase-field $\phi$ represents the two phases $-1$ and $1$, $\mu$ the variational derivative of the energy with respect to $\phi$, $\u$ the displacement of the elastic body, $\theta$ volumetric fluid content, and $p$ the associated pressure. The appearing functions will be specified later on.
The above system has been derived in \cite{STORVIK2022} as a generalised gradient flow of the energy functional
\begin{align} \label{eq:energy}
		\mathcal{F} (\phi, \E(\u), \theta ) &= \int_\Omega \tfrac{\gamma}{2} |\nabla\phi|^2 + \Psi(\phi) \dx 
		+ \int_\Omega W(\phi, \E(\u)) \dx 
		+ \int_\Omega  \tfrac{M(\phi)}{2} (\theta - \alpha(\phi) \div(\u))^2 \dx,
\end{align}   
where $\Omega$ is assumed to be a bounded domain in $\mathbb{R}^d$, $d\in\{2,3\}$, with sufficiently smooth boundary.
The model's ability to account for changes in material properties due to phase transitions makes it particularly relevant for biomedical applications, such as modelling the growth of malignant tumours, which exhibit altered interstitial fluid pressures and changes in the elastic properties of their surrounding matrix \cite{milosevictumor}. These biophysical changes can significantly impact tumour evolution and response to treatments \cite{lima2016, lima2017, tumorstresscheng, tumorstresshelminger, tumorstressstylianopoulos}. As done in tumour evolution equations of phase-field type \cite{garcke2022viscoelastic,garcke2021phase,fritz2023tumor}, we consider $\phi \in [-1,1]$ as the difference in volume fractions, that is, $\{\phi=1\}$ represents unmixed tumour tissue, while $\{\phi=-1\}$ represents surrounding healthy tissue. 

Previous studies have established the well-posedness of the Cahn-Hilliard-Biot model, proving the existence and uniqueness of weak solutions under various assumptions \cite{fritz2023wellposedness,abels2024existence,riethmuller2023wellposedness}. The next critical step is to develop numerical methods that preserve the inherent structures and properties of the model, ensuring stability and accuracy in simulations. Structure-preserving discretisations are numerical schemes designed to maintain key properties of the continuous model, such as energy dissipation, mass conservation, and thermodynamic consistency. These properties are crucial for obtaining physically realistic and stable numerical solutions, especially in long-term simulations of complex systems. Several studies have made significant contributions to structure-preserving discretisations with applications to general dissipative problems \cite{egger2019structure}, coupled Cahn-Hilliard equations \cite{brunk2022second,brunk2023second,brunk2023variational,brunk2024structure,shimura2020error}, Allen-Cahn equations \cite{bottcher2020structure,bendimerad2022structure,lan2023operator}, and poro-elasticity \cite{egger2021structure}. We note that in \cite{storvik2024sequential} a fully implicit scheme for the Cahn-Hilliard-Biot system is proposed. However, the focus was on efficient decoupling and not a structure-preserving approximation.

In this work, we propose and analyse a discrete scheme that preserves the gradient flow structure of the Cahn-Hilliard-Biot system. Our approach builds on recent advances in numerical analysis and aims to address the challenges posed by the coupled nonlinearities. We employ a combination of conforming finite element methods and implicit-explicit (IMEX) Euler time integration, tailored to respect the model's thermodynamic framework. By preserving the discrete energy laws and ensuring consistent approximations of the coupled variables, our methods achieve robust performance and accurate representations of the underlying physical processes.

Numerical methods for the related Cahn-Hilliard-Larché equations, which couple the Cahn-Hilliard equation with elasticity, have seen significant advances in recent years. We mention the fully practical finite element approximation showing stability and convergence in \cite{barrett2006finite}, the experimental observations in superalloys in \cite{blesgen2013cahn}, the finite difference schemes for the 1D system proving error estimates in \cite{shimura2020error,shimura2022new}, the adaptive mesh refinement and multigrid methods for efficient and reliable simulation in \cite{graser2014numerical}, the inclusion of stress-driven interface motion in \cite{garcke2007stress}, the convergence of a finite element and implicit Euler scheme in \cite{garcke2005CHENumerics}, the isogeometric analysis in topology optimisation problems in \cite{dede2012isogeometric}, and the posteriori error analysis in \cite{bartels2010posteriori}.

The structure of this article is as follows: In \Cref{sec:notation}, we introduce the variational formulation and the relevant assumptions of the model. In \Cref{sec:discretisation}, we detail the proposed structure-preserving discretization method and state the main theorem of this work, that is, there is a discrete solution that conserves the mass and energy-dissipation balances. \Cref{sec:proof} provides a rigorous proof of the proposed theorem. Finally, \Cref{sec:simulations} showcases numerical experiments that validate our theoretical findings and demonstrate the practical effectiveness of our methods in simulating tumour growth.

\section{Notation and variational formulation} \label{sec:notation}

Throughout the text, we use standard notation for function spaces and norms; see, for instance, \cite{Wloka}. For ease of notion, we write $\langle \cdot,\cdot \rangle$ for the $L^2(\Omega)$-inner product, $\|\cdot\|_{s,p}$, $s \in \mathbb{R}$, $p \geq 1$, for the norm of the Sobolev space $W^{s,p}(\Omega)$, and shortly $\|\cdot\|_{s}:=\|\cdot\|_{s,2}$.

Using the energy functional $\mathcal{F}$, see \eqref{eq:energy}, we can reveal the gradient structure of the system. Indeed, by recognising $\mu,p$ and partly $\sigma$ as  variational derivatives of $F$, we find:
\begin{align}
 \dt \phi  &- \div\left( m(\phi ) \nabla \mu \right) = r(\phi, \E(\u), \theta), & \mu &= \frac{\delta \mathcal{F}}{\delta\phi}, \label{eq:abs1}\\
 \div(\bm\sigma)&=\frac{\delta \mathcal{F}}{\delta\E(\u)} + \div(\CC_\nu(\phi)\E(\dt\u))  = \mathbf{f}, \label{eq:abs2}\\
 \dt\theta &-  \div(\kappa(\phi) \nabla p) = s(\phi, \E(\u), \theta), & p& = \frac{\delta \mathcal{F}}{\delta\theta}. \label{eq:abs3}
\end{align}

We make the following assumptions and choices: \medskip

\begin{enumerate}[start=0,label=\textup{(A\arabic*)}, ref=A\arabic*, leftmargin=1.5cm] 
\item \label{Ass:Initial} $\Omega \subset \mathbb{R}^d$, $d\in \{2,3\}$, is Lipschitz continuous, and we consider the boundary conditions
\begin{equation*}
    \restr{\u}{\partial\Omega} = \mathbf{0}, \quad \restr{\kappa(\phi)\nabla p\cdot\mathbf{n}}{\partial\Omega}= \restr{\nabla\phi\cdot\mathbf{n}}{\partial\Omega}=\restr{m(\phi)\nabla\mu\cdot\mathbf{n}}{\partial\Omega}=0; 
\end{equation*}
\item \label{Ass:Gamma} the interface parameter $\gamma$ is a positive constant;
\item \label{Ass:Pot} the Cahn-Hilliard potential $\Psi\in C^2(\mathbb{R})$ is bounded from below and admits a decomposition into a convex and concave part $\Psi_{\text{vex}},\Psi_{\text{cav}}$; furthermore, we assume that $|\Psi(x)|\leq c_1|x|^p + c_2$ for $p\leq 6$ and non-negative constants $c_1$, $c_2$;
\item \label{Ass:Elastic} the elastic energy $W$ is given by $W(x,\G)=(\G-\T(x)):\CC(x):(\G-\T(x))$ for $x \in \R$, $\G \in \R^{d\times d}$ with symmetric eigenstrain $\T \in C^{2}(\R)$ and symmetric elastic tensor $\CC \in C^{2}(\R)$  such that
\begin{equation*}
    \G:\CC(x):\G \geq c\snorm{G}^2, \qquad \mathcal{D}:\CC(x):\G = \mathcal{D}:\CC(x)\G =  \CC(x)\mathcal{D}:\G,
\end{equation*}
for any $x \in \R$, $\mathcal{D}, \G \in \R^{d\times d}_{\text{sym}}$; for simplicity, let us assume that $\mathcal{T}(x)=\xi(x-\alpha)\mathbf{I}$ for suitable constants $\xi$, $\alpha$; furthermore, we assume that all derivatives of $\CC$ are uniformly bounded from above and below.
\item \label{Ass:Visco} the modulus of viscoelasticity $\CC_\nu(\phi)$ fulfils the same assumptions as $\CC(\phi)$;
\item \label{Ass:Mob} the diffusion coefficients $m$ and $\kappa$ are positively bounded from above and below, that is, there exist positive constants $m_0,m_1,\kappa_0,\kappa_1$ such that $m_0\leq m(x)\leq m_1$ and $\kappa_0\leq \kappa(x)\leq \kappa_1$ for any $x \in \R$;
\item \label{Ass:Malpha} the functions $M$ and $\alpha$ are positively bounded from above and below, that is, there exist positive constants $M_0,M_1,\alpha_0,\alpha_1$ such that $M_0\leq M(x)\leq M_1$ and $\alpha_0\leq \alpha(x)\leq \alpha_1$ for any $x \in \R$; further, the functions are sufficiently regular with bounded derivatives, that is, $\snorm{\alpha^{(k)}(x)}\leq \alpha_{k+1}$ and $\snorm{M^{(k)}(x)}\leq M_{k+1}$, $k\in\{1,2\}$, for any $x \in \R$;
\item \label{Ass:Force} $\mathbf{f}\in L^2(0,T;H^{-1}(\Omega))$ and the source terms $r,s\in C^1(\mathbb{R} \times \mathbb{R}^{d\times d} \times \mathbb{R})$ are uniformly bounded. \medskip
\end{enumerate}

In the following, we will denote by $F(\phi,\E(\u),\theta)$ the density of the energy functional $\mathcal{F}$. We introduce the reduced energy density, that is, the energy density for the poro-elastic subsystem, via 
\begin{equation}
 \widetilde W(\phi,\E(\u),\theta):= W(\phi,\E(\u)) + \frac{M(\phi)}{2}\big(\theta-\alpha(\phi)\tr{\E(\u)}\big)^2.   \label{eq:tildeW}
\end{equation} 
Note that in view of the above assumptions, we observe that $\widetilde W$ is convex in $(\E(\u),\theta)$ for every fixed $\phi\in\mathbb{R}$. Furthermore, we will denote the partial derivatives of the energy density $F$ by $F_\phi,F_{\nabla\phi},F_{\E(\u)},F_{\theta}$ and similarly for the reduced energy density $\widetilde W$.

We remark that one can also use Dirichlet boundary conditions for $p$ or mixing both sets of conditions by introducing a Dirichlet and a Neumann boundary part, as done in \cite{abels2024existence}. Furthermore, for the analysis, the term $\div(\CC_{\nu}\E(\u))$ can also be replaced by $\div(c_{\nu}\div(\u))$, see \cite{riethmuller2023wellposedness} for existence results which such a regularisation. \smallskip

\noindent\textbf{Variational formulation:}
The above system can be directly converted into a variational formulation, which reads as follows:

\begin{align}
 \la \dt\phi,\psi \ra &+ \la m(\phi)\nabla\mu,\nabla\psi \ra  = \la r(\phi, \E(\u), \theta),\psi\ra, \label{eq:var1}\\
 \la \mu,\xi \ra &= \la F_{\nabla\phi}(\nabla\phi),\nabla\xi \ra + \la F_\phi(\phi, \E(\u), \theta),\xi \ra,  \label{eq:var2}\\
 \la F_{\E(\u)}(\phi, \E(\u), \theta) &+ \CC_\nu(\phi)\E(\dt\u),\E(\vv) \ra = \la \mathbf{f},\vv \ra,  \label{eq:var3}\\
 \la \dt\theta,\chi\ra &+ \la \kappa(\phi)\nabla p,\nabla\chi \ra = \la s(\phi, \E(\u), \theta), \chi\ra,  \label{eq:var4}\\
 \la p,q \ra &= \la F_\theta(\phi, \E(\u), \theta),q \ra,  \label{eq:var5}
\end{align}
where $\psi$, $\xi$, $\vv$, $\chi$, $q$ are suitable test functions.
The energy dissipation law of the system can then be obtained by choosing specific test functions.

\begin{lemma} \label{lem:balance} Sufficiently regular solutions of the system \eqref{eq:abs1}--\eqref{eq:abs3} satisfy the variational formulations \eqref{eq:var1}--\eqref{eq:var5}. Furthermore, the balance of mass and volumetric fluid content as well as the energy-dissipation balance hold, which are given by
  \begin{align*}
   \ddt\la \phi,1 \ra &= \la r(\phi, \E(\u), \theta),1 \ra, \\ \ddt\la \theta,1 \ra &= \la s(\phi, \E(\u), \theta),1 \ra, \\
   \ddt\mathcal{F}(\phi,\E(\u),\theta) &= -\mathcal{D}_\phi(\mu,\dt\u,p)  + \mathcal{P}_{\phi,\E(\u),\theta}(\mu,\dt\u,p),
  \end{align*}  
  with dissipation rate
  \begin{equation*}
  D_\phi(\mu,\dt\u,p):=  \la m(\phi)\nabla\mu,\nabla\mu \ra  + \la \CC_\nu(\phi)\E(\dt\u),\E(\dt\u) \ra +\la \kappa(\phi)\nabla p,\nabla p \ra,
  \end{equation*}
  and source rate
  \begin{equation*}
     \mathcal{P}_{\phi,\E(\u),\theta}(\mu,\dt\u,p) := \la r(\phi, \E(\u), \theta),\mu \ra + \la \mathbf{f},\dt\u \ra + \la s(\phi, \E(\u), \theta),p \ra.  
  \end{equation*}
\end{lemma}
\begin{proof}
First, we note that the variational form can be deduced by multiplying the system \eqref{eq:abs1}--\eqref{eq:abs3} with sufficiently regular test functions and applying the usual integration by parts. Regarding the balance of the mass and volumetric fluid content, we insert the test functions $\psi=1$ and $\chi=1$ into the variational forms \eqref{eq:var1} and \eqref{eq:var4}, respectively. In fact, this procedure immediately yields the desired results.
For the energy-dissipation balance, we compute the total derivative of the energy functional $\mathcal{F}$, see \eqref{eq:energy}, 
\begin{align*}
 \ddt \mathcal{F}(\phi, \E(\u), \theta )  &= \lla \frac{\partial F}{\partial\nabla\phi},\nabla\dt\phi \rra + \lla \frac{\partial F}{\partial\phi},\dt\phi \rra  + \lla \frac{\partial F}{\partial\E(\u)} ,\E(\dt\u) \rra + \lla \frac{\partial F}{\partial\theta},\dt\theta \rra. \\
 \intertext{Taking the test functions $\xi=\dt\phi$ and $q=\dt\theta$ in the variational forms \eqref{eq:var2} and \eqref{eq:var5}, respectively, further yields}
 & = \la \mu,\dt\phi \ra + \la F_{\E(\u)} ,\E(\dt\u) \ra + \la p,\dt\theta \ra.
 \intertext{Using the variational formulation \eqref{eq:var1}, \eqref{eq:var3} and \eqref{eq:var4} with the test functions $\psi=\mu$, $\vv=\dt\u$, $\chi=p$ finally yields}
 &= - \la m(\phi)\nabla\mu,\nabla\mu \ra  - \la \CC_\nu(\phi)\E(\dt\u),\E(\dt\u) \ra - \la \kappa(\phi)\nabla p,\nabla p \ra \\
 &\quad + \la r(\phi, \E(\u), \theta),\mu \ra + \la \mathbf{f},\dt\u \ra + \la s(\phi, \E(\u), \theta),p \ra,
\end{align*}
which proves the desired energy-dissipation balance.
 \end{proof}

\section{Structure-preserving discretisation} \label{sec:discretisation}

 As a preparatory step, we introduce the relevant notation and assumptions regarding the discretisation strategy for the Cahn-Hilliard-Biot system.

For spatial discretisation, we require that $\Th$ is a geometrically conforming partition of $\Omega$ into simplices. We denote the space of continuous, piecewise linear functions over $\Th$ and the space with zero Dirichlet data by
\begin{align*}
    \Vh &:= \{v \in H^1(\Omega)\cap C^0(\bar\Omega) : v|_K \in P_1(K) \quad \forall K \in \Th\}, \\ \Xh &:= \{v \in \Vh : v\vert_{\partial\Omega}=0\},
\end{align*}
respectively.
%\subsection*{Time discretisation}
We divide the time interval $[0,T]$ into uniform steps with step size $\tau>0$ and introduce $$\Itau:=\{0=t^0,t^1=\tau,\ldots, t^{n_T}=T\},$$ where $n_T=\tfrac{T}{\tau}$ is the absolute number of time-steps. We denote by $\Pi^1_c(\Itau;X)$ and $\Pi^0(\Itau;X)$ the spaces of continuous piecewise linear and piecewise constant functions on $\Itau$ with values in the space $X$, respectively. By $g^{n+1}$ and $g^n$ we denote the evaluation/approximation of a function $g$ in $\Pi^1_c(\Itau)$ or $\Pi^0(\Itau)$ at $t=\{t^{n+1},t^n\}$, respectively, and write $I_n=(t^n,t^{n+1})$. We introduce the following operators and abbreviations:
\begin{itemize}
    \item The time difference and the discrete time derivative, respectively, are denoted by
\begin{equation*}
	d^{n+1}g := g^{n+1} - g^n, \qquad d^{n+1}_\tau g := \frac{g^{n+1}-g^n}{\tau}.
\end{equation*}
\item The time average of a function $g(\rho)$ for $\rho\in\Pi^1_c(\Itau)$ is given by
\begin{equation*}
 g^{av}(\rho) := \frac{1}{\tau}\int_{I_n} g(\rho(s)) \;\mathrm{d}s.   
\end{equation*}
\item The convex-concave splitting of the Cahn-Hilliard potential $\Psi$ is denoted by
\begin{equation*}
\Psi'(\phi_h^{n+1},\phi_h^n):=\Psi'_{\text{vex}}(\phi_h^{n+1}) + \Psi'_{\text{cav}}(\phi_h^{n}).
\end{equation*}
\end{itemize}

%\subsection{Discrete scheme}

Next, we state the fully discrete scheme of the Cahn-Hilliard-Biot system that we analyse in this work.

\begin{problem}\label{prob:ac2}
	Let the initial data $(\phi_{h,0},\u_{h,0},\theta_{h,0})\in \Vh\times\Xh^{d}\times\Vh$ be given. Find $(\phi_h,\u_h,\theta_h)\in \Pi^1_c(\Itau;\Vh\times\Xh^{d}\times\Vh)$ and $(\mu_{h},p_h)\in \Pi^0(\Itau;\Vh\times\Vh)$ that satisfy the variational system
	\begin{align}
        \la\dtau\phi_h,\psi_h\ra & + \la m(\phi_h^n)\nabla\mu_h^{n+1},\nabla\psi_h \ra = \la r(\phi_h^{n},\E(\u_h^{n+1}),\theta_h^{n+1}),\psi_h \ra, \label{eq:pg1} \\
        \la \mu_h^{n+1}, \xi_h\ra &- \gamma\la \nabla\phi_h^{n+1},\nabla\xi_h \ra -  \la \Psi'(\phi_h^{n+1},\phi_h^n) + \widetilde W_\phi^{av}(\phi_h,\E(\u_h^{n+1}),\theta_h^{n+1}),\xi_h \ra = 0, \label{eq:pg2}\\
        \la \CC_{\nu}(\phi_h^n)\E(&\dtau\u_h),\E(\vv_h)\ra  + \la \widetilde W_{\E(\u)}(\phi_h^n,\E(\u_h^{n+1}),\theta_h^{n+1}) ,\E(\vv_h) \ra= \la \mathbf{f}^{n+1},\vv_h \ra, \label{eq:pg3}\\
        \la \dtau\theta_h ,\chi_h \ra &+ \la \kappa(\phi_h^n)\nabla p_h^{n+1} ,\nabla\chi_h  \ra = \la s(\phi_h^{n},\E(\u_h^{n}),\theta_h^{n}),\chi_h \ra, \label{eq:pg4}\\
        \la p_h^{n+1} ,q_h  \ra &- \la \widetilde W_\theta(\phi_h^n,\E(\u_h^{n+1}),\theta_h^{n+1}),q_h \ra = 0, \label{eq:pg5}
	\end{align}
	for every $(\psi_h,\xi_h,\vv_h,\chi_h,q_h)\in\Vh\times\Vh\times\Xh^d\times\Vh\times\Vh$ and every $0\leq n\leq n_T-1$.
\end{problem}

The properties of this scheme are condensed into the following theorem.

\begin{theorem} \label{thm:scheme}
Let \eqref{Ass:Initial}--\eqref{Ass:Force} hold. Furthermore, let $\tau\leq C_0$ for $C_0>0$ depend solely on the parameters and external forces. Then, for any $(\phi_{h,0},\u_{h,0},\theta_{h,0}) \in \Vh\times\Xh^d\times\Vh$, Problem~\ref{prob:ac2} admits at least one solution $(\phi_h,\mu_h,\u_h,\theta_h,p_h)$. Moreover, any such solution conserves the mass balances and energy-dissipation balance, that is, it holds
  \begin{align}
  \la \phi_h^{n},1 \ra = \la \phi_h^m,1 \ra &+ \tau\sum_{k=m}^{n-1}\la r(\phi_h^k,\E(\u_h^{k+1}),\theta_h^{k+1}),1 \ra , \label{thm1:mass:phi}  \\
  \la \theta_h^{n},1 \ra = \la \theta_h^m,1 \ra &+ \tau\sum_{k=m}^{n-1}\la s(\phi_h^{k},\E(\u_h^{k}),\theta_h^{k}),1 \ra , \label{thm1:mass:theta} \\
   \restrb{\mathcal{F}(\phi_h,\E(\u_h),\theta_h)}{t^m}{t^n} &+\int_{t^m}^{t^n}\mathcal{D}^*(\mu_h,\dt\u_h,p_h) \ds\leq  \int_{t^m}^{t^n} \mathcal{P}^*(\mu_h,\dt\u_h,p_h) \ds. \label{thm1:mass:energy}
 \end{align}
 % with numerical dissipation $\mathcal{D}_{\text{num},\hat\phi_h}(\dt\u_h,\dt\theta_h)\geq 0$ as given in \eqref{eq:numdiss} below. 
 The discrete dissipation and source rate in the above equation are given by
 \begin{align*}
   \int_{t^m}^{t^n}\mathcal{D}^*(\mu_h,\dt\u_h,p_h) \ds &:= \tau\sum_{k=m}^{n-1}\mathcal{D}_{\phi_h^{k}}(\mu_h^{k+1},d_\tau^{k+1}\u_h,p_h^{k+1}), \\
   \int_{t^m}^{t^n} \mathcal{P}^*(\mu_h,\dt\u_h,p_h) \ds &:= \tau\sum_{k=m}^{n-1} \la r(\phi_h^{k},\E(\u_h^{k+1}),\theta_h^{k+1}),\mu_h^{k+1} \ra \\
   &+ \la \mathbf{f}^{k+1},d_\tau^{k+1}\u_h \ra + \la s(\phi_h^{k},\E(\u_h^{k}),\theta_h^{k}),p_h^{k+1} \ra.
 \end{align*}
 For quasi-uniform meshes and under a CFL-type condition $2C_1\tau\leq h^{2d}$, the solution is unique. Here, $C_1>0$ depends on the parameters, the initial data, and the external forces.
 
\end{theorem}

The proof of this result will be considered in the next section. We note that the time-average strategy could also be applied to the whole system. The resulting scheme would preserve the energy-dissipation balance without numerical dissipation, cf.~\cite{BrunkCh} where this strategy was applied to the Cahn-Hilliard equation.

\begin{lemma}\label{lem:split}
The numerical discretisation of the Cahn-Hilliard-Biot system as stated in Problem \ref{prob:ac2} can be realised as a decoupled scheme where the first step is linear, while the second step is, in general, nonlinear. The decoupled scheme reads as follows:
\begin{problem}\label{prob:split}
	Let $(\phi_{h}^n,\u_{h}^n,\theta_{h}^n)\in \Vh\times\Xh^{d}\times\Vh$ be given. 
    \begin{enumerate}
        \item Find $(\u_h^{n+1},\theta_h^{n+1},p_h^{n+1})\in \Xh^d\times\Vh\times\Vh$ such that
	\begin{align}
        \la \CC_{\nu}(\phi_h^n)\E(&\dtau\u_h),\E(\vv_h)\ra  + \la \widetilde W_{\E(\u)}(\phi_h^n,\E(\u_h^{n+1}),\theta_h^{n+1}) ,\E(\vv_h) \ra= \la \mathbf{f}^{n+1},\vv \ra, \label{eq:dis1} \\
        \la \dtau\theta_h ,\chi_h \ra &+ \la \kappa(\phi_h^n)\nabla p_h^{n+1} ,\nabla\chi_h  \ra = \la s(\phi_h^{n},\E(\u_h^{n}),\theta_h^{n}),\chi_h\ra, \label{eq:dis2} \\
        \la p_h^{n+1} ,q_h  \ra &- \la \widetilde W_\theta(\phi_h^n,\E(\u_h^{n+1}),\theta_h^{n+1}),q_h \ra = 0, \label{eq:dis3}
	\end{align}
	holds for $(\vv_h,\chi_h,q_h)\in\Xh^d\times\Vh\times\Vh$.
	\item  Find $(\phi_h^{n+1},\mu_h^{n+1})\in \Vh\times\Vh$ such that
	\begin{align}
        \la\dtau\phi_h,\psi_h\ra & + \la m(\phi_h^n)\nabla\mu_h^{n+1},\nabla\psi_h \ra = \la r(\phi_h^{n},\E(\u_h^{n+1}),\theta_h^{n+1}),\psi_h \ra, \label{eq:dis4} \\
        \la \mu_h^{n+1}, \xi_h\ra &- \gamma\la \nabla\phi_h^{n+1},\nabla\xi_h \ra -  \la \Psi'(\phi_h^{n+1},\phi_h^n),\xi_h \ra \label{eq:dis5} \\
        &- \la \widetilde W^{av}_\phi(\phi_h,\E(\u_h^{n+1}),\theta_h^{n+1}),\xi_h \ra = 0, \notag
	\end{align}
	holds for $(\psi_h,\xi_h)\in\Vh\times\Vh$.
 \end{enumerate}
\end{problem}
\end{lemma}

We note that the above scheme can be used similarly to approximate the Cahn-Hilliard-Larch{\'e} system, that is, in the case $M\equiv 0$. Hence, we also provide a decoupling scheme for the Cahn-Hilliard-Larch{\'e} system.

\begin{proof}
 Careful inspection of equations \eqref{eq:pg1}--\eqref{eq:pg5} reveal that the poro-elastic subsystem  \eqref{eq:pg3}--\eqref{eq:pg5} does not depend on $\phi_h^{n+1}$ but only on $\phi_h^n$. Hence, one can solve this subsystem and use the obtained solutions $\u_h^{n+1}$ and $\theta_h^{n+1}$ to solve the Cahn-Hilliard subsystem \eqref{eq:pg1}--\eqref{eq:pg2}. Note that, due to the special structure of the energy, $\widetilde W_{\E(\u)}(\phi_h^n,\E(\u_h^{n+1}),\theta_h^{n+1})$ and $\widetilde W_{\theta}(\phi_h^n,\E(\u_h^{n+1}),\theta_h^{n+1})$ are linear in $\E(\u_h^{n+1})$ and $\theta_h^{n+1}$.
\end{proof}

    We note that the discretisation of the source terms $r$ and $s$ is chosen as above for specific reasons. First, the explicit choice in $s$ leads to a linear discretisation for the poro-elastic system. Second, the explicit evaluation of $\phi$ in $r$ is chosen to prove the uniqueness of discrete solutions. Indeed, for an implicit evaluation the uniqueness proof does not work, and one has to resort to other approaches, as done in \cite{GarckeTrautwein2022} for a Cahn-Hilliard equation that is coupled to a reaction-diffusion equation.

\section{Proof of Theorem 1}
\label{sec:proof}

In this section, we prove \cref{thm:scheme} which states the existence of a solution to Problem \ref{prob:ac2} (and equivalently Problem \ref{prob:split} by \cref{lem:split}), the conservation of mass and energy-dissipation balance and the uniqueness of the solution for a CFL-type condition. We begin by proving the balance laws \eqref{thm1:mass:phi}--\eqref{thm1:mass:energy}.

\subsection{Balance laws}

Regarding the balance of mass for $\phi_h$, see \eqref{thm1:mass:phi}, we insert the test function $\psi_h=1$ in the equation that governs $\phi_h$ in the decoupled scheme; see \eqref{eq:dis5}. Thus, we obtain
\begin{align*}
\restrb{\la \phi_h, 1\ra}{t^n}{t^{n+1}} = \int_{I_n}\la\dt\phi_h, 1\ra &= \int_{I_n} \la \dtau\phi_h,1 \ra \\
&= - \int_{I_n}\lla m(\phi_h^n)\nabla\mu_h^{n+1},\nabla 1 \rra + \int_{I_n}\lla r(\phi_h^{n},\E(\u_h^{n+1}),\theta_h^{n+1}),1 \rra \\
&= \tau\lla r(\phi_h^{n},\E(\u_h^{n+1}),\theta_h^{n+1}),1 \rra.\end{align*}
Summation over the relevant time-steps yields the result \eqref{thm1:mass:phi}. With the same computations using $\chi_h=1$ in \eqref{eq:dis2}, we obtain balance of volumetric fluid content \eqref{thm1:mass:theta}, that is,
\begin{align*}
\restrb{\la \theta_h, 1\ra}{t^n}{t^{n+1}} &= \int_{I_n} \la \dtau\theta_h,1 \ra =\tau\la s(\phi_h^{n},\E(\u_h^{n}),\theta_h^{n}),1 \ra.
\end{align*}
  Let us now prove the energy-dissipation balance as stated in \eqref{thm1:mass:energy}. We have
 \begin{align*}
\frac{1}{\tau}\restrb{\mathcal{F}(\phi_h,\E(\u_h),\theta_h)}{t^n}{t^{n+1}} & = \gamma\la \nabla\phi^{n+1},\dtau\nabla\phi_h \ra + \la \Psi(\phi_h^{n+1},\phi_h^n),\dtau\phi_h \ra \\
&+ \frac{1}{\tau}\la \widetilde W(\phi_h^{n+1},\E(\u_h^{n+1}),\theta_h^{n+1}) - W(\phi_h^{n},\E(\u_h^{n}),\theta_h^{n}),1 \ra\\
&+ \frac{\tau\gamma}{2}\norm{\nabla\dtau\phi_h}_0^2 + \frac{1}{\tau}\la \Psi(\phi_h^{n+1}) - \Psi(\phi_h^{n}) -\Psi(\phi_h^{n+1},\phi_h^n)d^{n+1}\phi_h , 1\ra \\
&= (a) + (b) + (c) + (d) + (e), 
\end{align*}
where we introduced the short notation $(a)$--$(e)$ for the five terms on the right-hand side. We consider each term and try to simplify them. \medskip

We start with the third inner product called $(c)$ by adding $\pm\la \widetilde W(\phi_h^n,\E(\u_h^{n+1}),\theta_h^{n+1}),1 \ra$ we find
\begin{align*}
 (c) & = \frac{1}{\tau}\la \widetilde W(\phi_h^{n+1},\E(\u_h^{n+1}),\theta_h^{n+1}) - W(\phi_h^{n},\E(\u_h^{n+1}),\theta_h^{n+1}),1 \ra \\ &+ \frac{1}{\tau}\la \widetilde W(\phi_h^{n},\E(\u_h^{n+1}),\theta_h^{n+1}) - W(\phi_h^{n},\E(\u_h^{n}),\theta_h^{n}),1 \ra \\
 & = \frac{1}{\tau}\int_{I_n} \dt \widetilde W(\phi_h,\E(\u_h^{n+1}),\theta_h^{n+1}) \ds  +\frac{1}{\tau} \la \widetilde W(\phi_h^{n},\E(\u_h^{n+1}),\theta_h^{n+1}) - W(\phi_h^{n},\E(\u_h^{n}),\theta_h^{n}),1 \ra.
\end{align*}
Computing the time derivative and adding $\pm\la  \widetilde W_{\E(\u)}(\phi_h^{n},\E(\u_h^{n+1}),\theta_h^{n+1}),\E(\dtau \u_h)\ra$ \linebreak and $\pm\la  \widetilde W_{\theta}(\phi_h^{n},\E(\u_h^{n+1}),\theta_h^{n+1}),\dtau \theta_h\ra$ yields
\begin{align*}
 (c) & = \frac{1}{\tau}\int_{I_n}\la \widetilde W_\phi(\phi_h,\E(\u_h^{n+1}),\theta_h^{n+1}),\dt\phi_h \ra \ds + \la  \widetilde W_{\E(\u)}(\phi_h^{n},\E(\u_h^{n+1}),\theta_h^{n+1}),\E(\dtau \u_h)\ra  \\
 &+ \la  \widetilde W_{\theta}(\phi_h^{n},\E(\u_h^{n+1}),\theta_h^{n+1}),\dtau \theta_h\ra\\
 & + \frac{1}{\tau}\la \widetilde W(\phi_h^{n},\E(\u_h^{n+1}),\theta_h^{n+1}) - \widetilde W(\phi_h^{n},\E(\u_h^{n}),\theta_h^{n}) - \widetilde W_{\E(\u)}(\phi_h^{n},\E(\u_h^{n+1}),\theta_h^{n+1})\E(d^{n+1} \u_h) \\
 &- \widetilde W_{\theta}(\phi_h^{n},\E(\u_h^{n+1}),\theta_h^{n+1}),d^{n+1}\theta_h,1 \ra \\
 & = (i) + (ii) + (iii) + (iv),
\end{align*}
where we introduced the short notation $(i)$--$(iv)$ for the terms on the right-hand side. 

We note that since the time derivative is piecewise-constant on $I_n$ and coincides with the discrete time derivative $\restr{\dt\phi_h}{I_n}=\dtau\phi_h$, we can expand $(i)$ as 
\begin{align*}
 (i) = \lla \frac{1}{\tau}\int_{I_n} \widetilde W_\phi(\phi_h,\E(\u_h^{n+1}),\theta_h^{n+1}) \ds,\dtau\phi_h \rra = \la W^{av}_\phi(\phi_h,\E(\u_h^{n+1}),\theta_h^{n+1}), \dtau\phi_h\ra.
\end{align*}

The addition of $(a)$, $(b)$, $(i)$, $(ii)$, $(iii)$ yields
\begin{align*}
(A) & =  \gamma\la \nabla\phi^{n+1},\dtau\nabla\phi_h \ra + \la \Psi(\phi_h^{n+1},\phi_h^n),\dtau\phi_h \ra + \la W^{av}_\phi(\phi_h,\E(\u_h^{n+1}),\theta_h^{n+1}), \dtau\phi_h\ra  \\
& +\la  \widetilde W_{\E(\u)}(\phi_h^{n},\E(\u_h^{n+1}),\theta_h^{n+1}),\E(\dtau \u_h)\ra  + \la  \widetilde W_{\theta}(\phi_h^{n},\E(\u_h^{n+1}),\theta_h^{n+1}),\dtau \theta_h\ra. 
\end{align*}

The addition of the remaining terms, that is, $(d)$, $(e)$ and $(iv)$, yields
\begin{align*}
 (B) & = \frac{\tau\gamma}{2}\norm{\nabla\dtau\phi_h}_0^2 + \frac{1}{\tau}\la \Psi(\phi_h^{n+1}) - \Psi(\phi_h^{n}) -\Psi(\phi_h^{n+1},\phi_h^n)d^{n+1}\phi_h , 1\ra\\
 & + \frac{1}{\tau}\la \widetilde W(\phi_h^{n},\E(\u_h^{n+1}),\theta_h^{n+1}) - \widetilde W(\phi_h^{n},\E(\u_h^{n}),\theta_h^{n}) - \widetilde W_{\E(\u)}(\phi_h^{n},\E(\u_h^{n+1}),\theta_h^{n+1})\E(d^{n+1} \u_h) \\
 &- \widetilde W_{\theta}(\phi_h^{n},\E(\u_h^{n+1}),\theta_h^{n+1})d^{n+1} \theta_h,1 \ra.
\end{align*}

It remains to expand $(A)$ and $(B$). We start with $(A)$ and using as test functions $\xi_h=\dtau\phi_h\in\Vh$ in \eqref{eq:dis5} and $q_h=\dtau\theta_h\in\Vh$ in \eqref{eq:dis3} yields
\begin{align*}
 (A) &= \la \mu_h^{n+1},\dtau\phi_h \ra + \la  \widetilde W_{\E(\u)}(\phi_h^{n},\E(\u_h^{n+1}),\theta_h^{n+1}),\E(\dtau \u_h)\ra + \la p_h^{n+1},\dtau\theta_h \ra.
 \intertext{Insertion of $\psi_h=\mu_h^{n+1}\in\Vh$ in \eqref{eq:dis4}, $\vv_h=\dtau \u_h\in\Xh^d$ in \eqref{eq:dis1} and $\chi_h=p_h^{n+1}\in\Vh$ in \eqref{eq:dis2} yields}
 & = - \la m(\phi_h)\nabla\mu_h^{n+1},\nabla\mu_h^{n+1} \ra - \la \CC(\phi_h^n)\E(\dtau \u_h),\E(\dtau \u_h)\ra - \la \kappa(\phi_h^n)\nabla p_h^{n+1},\nabla p_h^{n+1}\ra \\
 & + \la r^{n,n+1},\mu_h^{n+1} \ra + \la \mathbf{f},\dtau\u_h \ra + \la s^{n},p_h^{n+1} \ra \\
 & = \mathcal{D}_{\phi_h^n}(\mu_h^{n+1},\dtau \u_h,p_h^{n+1}) + \int_{t^n}^{t^{n+1}}\mathcal{P}^*(\mu_h^{n+1},\dtau\u_h,p_h^{n+1}),
\end{align*}
where we defined the abbreviations $r^{n,n+1}:=r(\phi_h^{n},\E(\u_h^{n+1}),\theta_h^{n+1})$ and $s^n:=s(\phi_h^{n},\E(\u_h^{n}),\theta_h^{n})$. Furthermore, we used the definitions of the dissipation and source rates as introduced in \cref{lem:balance} and \cref{thm:scheme}.

Hence, $(A)$ corresponds to the usual variational structure of the problem and produces the dissipation and source rates. Next, we turn to $(B)$ and by utilising Taylor's expansion, we find 
\begin{align}
  (B) &= -\frac{\tau\gamma}{2}\norm{\nabla\dtau\phi_h}_0^2 - \frac{\tau}{2}\la (\Psi_{vex}''(z_1)-\Psi_{cav}''(z_2))\dtau\phi_h, \dtau\phi_h\ra \label{eq:numdiss}\\
 & - \frac{\tau}{2}\la \mathbf{H}_{\widetilde W,\E(\u),\theta}(\phi_h^n,\mathbf{z}_3)(\E(\dtau\u_h),\dtau\theta_h)^\top,(\E(\dtau\u_h),\dtau\theta_h)^\top \ra, \notag
\end{align}
where $z_1,z_2$ are convex combinations of $\phi_h^n,\phi_h^{n+1}$, while $\mathbf{z}_3$ is a convex combination of $(\E(\u_h^{n+1}),\theta_h^{n+1})^\top$ and $(\E(\u_h^{n}),\theta_h^{n})^\top$. Furthermore, $\mathbf{H}_{\widetilde W,\E(\u),\theta}$ denotes the Hessian of $\widetilde W$ with respect to $\E(\u)$ and $\theta$. We recall that $\widetilde W$, cf.~\eqref{eq:tildeW}, is convex, and thus yields $(B) \leq 0$. This finishes the energy-dissipation balance, see \eqref{thm1:mass:energy}.

\subsection{A priori bounds}
To prove the existence of a solution to Problem \ref{prob:ac2}, we derive a priori bounds from the energy-dissipation balance.

Recalling the energy-dissipation balance \eqref{thm1:mass:energy} and using the assumptions \eqref{Ass:Elastic}--\eqref{Ass:Malpha} on the lower bounds of the system's functions, we find 
\begin{align*}
   \mathcal{F}(\phi_h^{n+1},\E(\u_h^{n+1}),\theta_h^{n+1})&+\tau\left(m_0\norm{\nabla\mu_h^{n+1}}_0^2 + c\norm{\E(\dtau\u_h)}_0^2 + \kappa_0\norm{\nabla p_h^{n+1}}_0^2 \right) \\
   &\leq \mathcal{F}(\phi_h^{n},\E(\u_h^{n}),\theta_h^{n}) + \tau(\la r^{n,n+1},\mu_h^{n+1} \ra + \la \mathbf{f}^{n+1},\E(\dtau\u_h)\ra + \la s^n,p_h^{n+1} \ra ).
 \end{align*}

Next, we estimate a lower bound for $\mathcal{F}(\phi_h^{n+1},\E(\u_h^{n+1}),\theta_h^{n+1})$ and an upper bound for $\mathcal{F}(\phi_h^{n},\u_h^{n},\theta_h^{n})$. First, we find
\begin{align*}
 \mathcal{F}(\phi_h^{n+1},\E(\u_h^{n+1}),\theta_h^{n+1}) &\geq \frac{\gamma}{2}\norm{\nabla\phi_h^{n+1}}_0^2 + C_\u\norm{\u_h^{n+1}}_1^2 - C(\norm{\phi_h^{n+1}}_0^2 +1) \\ &+ \frac{M_0}{2}\norm{\theta_h^{n+1}-\alpha(\phi_h^{n+1})\div(\u_h^{n+1})}_0^2. 
\end{align*}
 Following the computations in \cite{abels2024existence}, we can estimate the right-hand side further to obtain 
\begin{align*}
 \mathcal{F}(\phi_h^{n+1},\E(\u_h^{n+1}),\theta_h^{n+1}) &\geq \frac{\gamma}{2}\norm{\nabla\phi_h^{n+1}}_0^2 + \Big(C_\u-\frac{M_0}{2}\alpha_0^2\big(\frac{1}{\delta_\theta}-1\big)\Big)\norm{\u_h^{n+1}}_1^2 - C(\norm{\phi_h^{n+1}}_0^2 +1) \\
 &+ \frac{M_0}{2}(1-\delta_\theta)\norm{\theta^{n+1}_h-\alpha(\phi_h^{n+1})\div(\u_h^{n+1})}_0^2. 
\end{align*}
Regarding the upper bound of $\mathcal{F}(\phi_h^{n},\E(\u_h^{n}),\theta_h^{n})$, we compute
\begin{align*}
 \mathcal{F}(\phi_h^{n},\E(\u_h^{n}),\theta_h^{n}) \leq  \frac{\gamma}{2}\norm{\nabla\phi_h^{n}}_0^2 + \norm{\Psi(\phi_h^{n})}_{0,1}^2 + C\norm{\u_h^n}_1^2 + C(\norm{\phi_h^n}_0^2 +1) + C\norm{\theta^{n}_h}_0^2. 
 \end{align*}

 Next, we insert $\psi_h=\beta\phi_h^{n+1}$ in \eqref{eq:dis4}, which gives
 \begin{equation*}
  \frac{\beta}{2}\norm{\phi_h^{n+1}}_0^2 + \frac{\beta}{2}\norm{\dtau\phi_h}_0^2 \leq \frac{\beta}{2}\norm{\phi_h^{n}}_0^2 + \tau m_0\norm{\nabla\mu_h^{n+1}}_0\norm{\nabla\phi_h^{n+1}}_0 + \tau\norm{r^{n,n+1}}_0\norm{\phi_h^{n+1}}_0.   
 \end{equation*}

 Combining the above estimates and choosing $\delta_\theta\leq 1-\varepsilon$ and $\beta \geq 2C$, we find
 \begin{equation}\label{Eq:Estimate}\begin{aligned}
  C_a\big(\norm{\phi_h^{n+1}}_1^2 + \norm{\u_h^{n+1}}_1^2 &+ \norm{\theta_h^{n+1}}_0^2\big) +\tau\left(m_0\norm{\nabla\mu_h^{n+1}}_0^2 + c_0\norm{\dtau\u_h}_1^2 + \kappa_0\norm{\nabla p_h^{n+1}}_0^2 \right) \\
  &\leq C_b\big(\norm{\phi_h^{n}}_1^2 + \norm{\u_h^n}_1^2 + \norm{\theta_h^{n}}_0^2\big)  \\
  &+ \tau\big(\la r^{n,n+1},\mu_h^{n+1} \ra + \la \mathbf{f}^{n+1},\dtau\u_h\ra + \la s^n,p_h^{n+1} \ra \big).
 \end{aligned}\end{equation}

It remains to treat the production terms and the external forces on the right-hand side of the inequality \eqref{Eq:Estimate}. First, we estimate the forces as follows:
\begin{align*}
 \la \mathbf{f}^{n+1},\dtau\u_h\ra \leq \norm{\mathbf{f}^{n+1}}_{-1}\norm{\dtau\u_h}_1 \leq \frac{c_0}{2}\norm{\dtau\u_h}_1^2 + C\norm{\mathbf{f}^{n+1}}_{-1}^2, \\
 %\la s^n,p_h^{n+1} \ra \leq \norm{s^n}_0\norm{p_h^{n+1}}_0 \leq \delta\kappa_0\norm{\nabla p_h^{n+1}}_0^2  + C\norm{s^n}_0^2.
 \la s^n,p_h^{n+1} \ra \leq \norm{s^n}_0\norm{p_h^{n+1}}_0 \leq C\norm{\theta^{n+1}_h}_0^2 + C\norm{\u^{n+1}_h}_1^2   + C\norm{s^n}_0^2.
\end{align*}
Furthermore, the production term in the Cahn-Hilliard equation gives
\begin{align*}
  \la r^{n,n+1},\mu_h^{n+1}\ra &= \la r^{n,n+1},\mu_h^{n+1}-\la \mu_h^{n+1},1 \ra\ra + \la r^{n,n+1},\la \mu_h^{n+1},1 \ra\ra  \\
  &\leq C(\delta)\norm{r^{n,n+1}}_{0}^2 + \delta m_0\norm{\nabla\mu_h^{n+1}}^2 + \delta\la \mu_h^{n+1},1\ra^2.
\end{align*}
Since we assumed that $r$ is bounded, see \eqref{Ass:Force}, it remains to estimate the squared mean value of $\mu_h^{n+1}$ in the last term of the inequality. This is done as follows:
$$\begin{aligned}
 \la \mu_h^{n+1},1 \ra^2 & \leq \norm{\Psi_{vex}'(\phi_h^{n+1})}_{0,1}^2 +\norm{\Psi_{cav}'(\phi_h^{n})}_{0,1}^2 + \norm{\widetilde W_\phi^{av}(\phi_h,\E(\u_h^{n+1}),\theta_h^{n+1})}_{0,1}^2   \\
 & \leq C\big(\norm{\phi_h^{n+1}}_1^2 + \norm{\phi_h^{n}}_1^2\big) + C\big(\norm{\widetilde W_\phi(\phi_h^n,\E(\u_h^{n+1}),\theta_h^{n+1})}_{0,1}^2 \\ &+ \norm{\widetilde W_\phi(\phi_h^{n+1},\E(\u_h^{n+1}),\theta_h^{n+1})}_{0,1}^2\big).
\end{aligned}$$
Next, we estimate the two norms on the right-hand side that involve the derivative of the functional $\widetilde W$ with respect to $\phi$. Using its structure, see \eqref{eq:tildeW}, we find for $\phi_h^*\in\{\phi_h^{n+1},\phi_h^n\}$ that
\begin{align*}
    \norm{\widetilde W_\phi(\phi_h^*,\E(\u_h^{n+1}),\theta_h^{n+1})}^2_{0,1} &\leq \norm{\E(\u_h^{n+1})-\mathcal{T}(\phi_h^*)}_0^2 + \norm{\mathcal{T}'(\phi_h^*)}_0^2 + C\norm{\theta_h^{n+1}}_0^2 + \norm{\u_h^{n+1}}_1^2 \\
    & \leq C\norm{\E(\u_h^{n+1})}_0^2 + C\norm{\phi_h^*}_0^2 + C\norm{\theta_h^{n+1}}_0^2 + \norm{\u_h^{n+1}}_1^2.
\end{align*}

Combining all estimates and inserting them back into \eqref{Eq:Estimate}, we find that
\begin{align*}
  (C_a-C\tau)\big(\norm{\phi_h^{n+1}}_1^2 + \norm{\u_h^{n+1}}_1^2 &+ \norm{\theta_h^{n+1}}_0^2\big) +\tau\left(m_0\norm{\nabla\mu_h^{n+1}}_0^2 + c_0\norm{\dtau\u_h}_1^2 + \kappa_0\norm{\nabla p_h^{n+1}}_0^2 \right) \\
  &\leq (C_b +C\tau)\big(\norm{\phi_h^{n}}_1^2 + \norm{\u_h^n}_1^2 + \norm{\theta_h^{n}}_0^2\big)  \\
  &+ C\tau\big(\norm{r^{n,n+1}}_0^2 + \norm{s^n}_0^2 + \norm{\mathbf{f}^{n+1}}_{-1}^2\big).  
 \end{align*}

At this point, we choose the step size $\tau \leq \frac{C_a}{2C}$, so that we can write the above problem as
\begin{equation*}
y^{n+1} + \tau D^{n+1} \leq Cy^n + \tau g^n,
\end{equation*}
where we have defined 
\begin{align*}
    y^k &:= \norm{\phi_h^{k}}_1^2 + \norm{\u_h^{k}}_1^2 + \norm{\theta^{k}}_0^2, \\
    D^{n+1}&:=m_0\norm{\nabla\mu_h^{n+1}}_0^2 + c_0\norm{\dtau\u_h}_1^2 + \kappa_0\norm{\nabla p_h^{n+1}}_0^2, \\
    g^n&:= C\big(\norm{r^{n,n+1}}_0^2 + \norm{s^n}_0^2 + \norm{\mathbf{f}^{n+1}}_{-1}^2\big).
\end{align*}
We sum over the first $n$ time-steps and apply an immediate consequence of the discrete Gronwall lemma, see \cite[Lemma 2.2]{bartels2015numerical}, which yields
\begin{align}
  \frac{C_a}{2}\big(\norm{\phi_h^{n+1}}_1^2 + \norm{\u_h^{n+1}}_1^2 &+ \norm{\theta^{n+1}}_0^2\big) +\sum_{k=0}^ {n-1}\tau\left(m_0\norm{\nabla\mu_h^{k+1}}_0^2 + c_0\norm{d_\tau^{k+1}\u_h}_1^2 + \kappa_0\norm{\nabla p_h^{k+1}}_0^2 \right) \notag\\
  &\leq C^*(\norm{\phi_h^{0}}_1^2 + \norm{\u^0_h}_1^2 + \norm{\theta^{0}_h}_0^2)  \label{eq:apriori}\\
  &+ \sum_{k=0}^{n-1}\tau C(\norm{r^{k,k+1}}_0^2 + \norm{s^k}_0^2 + \norm{\mathbf{f}^{k+1}}_{-1}^2). \notag 
 \end{align}
Due to assumption \eqref{Ass:Force}, the terms involving $r$, $s$ and $\mathbf{f}$ are uniformly bounded in the respective norms. Thus, we obtain the desired a priori bound.

 \subsection{Existence} \label{Subsec:Ex}
  We consider the $n$-th time-step and assume that $\phi_h^{n-1},\u_h^{n-1}$ and $\theta_h^{n-1}$ are already known. 
After choosing a basis for $\Vh\times\Vh\times\Xh^d\times\Vh\times\Vh$, we rewrite the variational system \eqref{eq:pg1}--\eqref{eq:pg5} as a nonlinear system in the form of $$J(x)=0 \text{ in } \mathbb{R}^{Z},$$ with $Z=\mathrm{dim(\Vh)}^4\times\mathrm{dim}(\Xh)^d$. In fact, writing $\la J(x),x\ra$ is equivalent to testing the corresponding variational identities with $$x=(\mu_h^{n+1}+\beta\phi_h^{n+1},\dtau\phi_h,\dtau\u_h,p_h^{n+1},\dtau\theta_h).$$
As a consequence of the latter, we thus obtain 
\begin{align*}
\la J(x),x\ra &\geq y^{n+1} - Cy^n + \tau D^{n+1} - \tau C.
\end{align*}
Using the same arguments as we have used to derive the a priori bound \eqref{eq:apriori}, we directly obtain $$\la J(x),x\ra \to \infty \text{ for } |x| \to \infty.$$ 
Thus, the existence of a solution follows from a corollary to Brouwer's fixed-point theorem, see \cite[Proposition~2.8]{Zeidler1}.

\subsection{Uniqueness}
We consider one time-step and recall Lemma \ref{lem:split}, that is, the scheme can be formulated as a splitting scheme. For the sake of presentation, we will neglect the space discretisation index $h$ in the remaining part of the proof.

Since we already know that solutions exist, see \cref{Subsec:Ex}, the poro-elastic subsystem has a unique solution as it is linear. Hence, the triple $(\u^{n+1},\theta^{n+1},p^{n+1})$ is uniquely determined for a given triple $(\phi^n,\u^n,\theta^n)$. Regarding the Cahn-Hilliard subsystem, we assume that two different solutions $(\phi_1^{n+1},\mu_1^{n+1})$ and $(\phi_2^{n+1},\mu_2^{n+1})$ exist. We denote their differences by $\phi^{n+1}$ and $\mu^{n+1}$. Taking the difference in the variational forms, it yields the following discrete formulations
\begin{align}
\la\phi^{n+1},\psi\ra & + \tau\la m(\phi_1^n)\nabla\mu^{n+1},\nabla\psi \ra = 0 \label{eq:diffphi}\\
\la \mu^{n+1}, \xi\ra &- \gamma\la \nabla\phi^{n+1},\nabla\xi \ra -  \la \Psi_{vex}''(\zeta)\phi^{n+1},\xi \ra - \la \widetilde W^{av}_{\phi\phi}(\zeta,\E(\u^{n+1}),\theta^{n+1})\phi^{n+1},\xi \ra = 0,\label{eq:diffmu}
\end{align}
which holds for all $(\psi,\xi)\in\Vh\times\Vh$.
Above, $\widetilde W^{av}_{\phi\phi}$ is defined as the difference of $\widetilde W^{av}_{\phi}$ evaluated at the two solution sets.
In addition, we note that $\zeta$ is a convex combination of $\phi_1^{n+1}$ and $\phi_2^{n+1}$.

In the following, we introduce the discrete inverse Laplacian with a mobility function in the spirit of \cite{Barrett1999}. We define the mean free space $\mathcal{G}:=\{v\in (H^1(\Omega))':\la v,1 \ra =0\}$ and the weighted inverse Laplacian $\Delta_{h,m}^{-1}:\mathcal{G}\to\Vh$, which is given by
\begin{align*}
  \la m(\phi_1^n)\nabla(-\Delta_{h,m}^{-1}v),\nabla w \ra = \la v,w \ra.
\end{align*}
This operator defines a discrete $H^{-1}(\Omega)$-norm via
\begin{equation*}
  \norm{v}_{-1,m}^2 := \la \nabla\Delta_{h,m}^{-1}v,\nabla \Delta_{h}^{-1}v\ra = \la m(\phi_1^n)\nabla(-\Delta_{h,m}^{-1}v),\nabla(-\Delta_{h,m}^{-1}v) \ra 
\end{equation*}
and an associated interpolation inequality is given by
\begin{equation*}
  \norm{v}_0^2 = \la v,v \ra = \la m(\phi_1^n)\nabla(-\Delta_{h,m}^{-1}v), \nabla v\ra \leq \norm{\sqrt{m(\phi_1^n)}}_{0,\infty}\norm{v}_{-1,h,m}\norm{\nabla v}_0.  
\end{equation*}

Although $\phi_1^{n+1}$ and $\phi_2^{n+1}$ are not mean free, the difference $\phi^{n+1}$ is mean free by construction; see \eqref{eq:diffphi} with $\psi=1$. Hence, we are allowed to apply the weighted inverse Laplacian on $\phi^{n+1}.$
Using $\psi=-\Delta_{h,m}^{-1}\phi^{n+1}\in\Vh$ and $\xi=\phi^{n+1}\in\Vh$ as test functions in \eqref{eq:diffphi} and \eqref{eq:diffmu}, respectively, which yields
\begin{align}
 \norm{\phi^{n+1}}_{-1,m}^2 &+ \tau\big(\norm{\nabla\phi^{n+1}}_0^2 + \la \Psi_{vex}''(\zeta)\phi^{n+1},\phi^{n+1} \ra\big) \label{eq:uniqgron}\\
 &= - \tau\la \widetilde W^{av}_{\phi\phi}(\zeta,\E(\u^{n+1}),\theta^{n+1})\phi^{n+1},\phi^{n+1} \ra \notag
\end{align}

It remains to estimate the right-hand side in a suitable manner. First, we observe that we can split $\widetilde W^{av}_{\phi\phi}$ on the right-hand side into two terms using the convex nature of $\zeta$ as follows
\begin{align*}
 \big|\widetilde W^{av}_{\phi\phi}(\zeta,\E(\u^{n+1}),\theta^{n+1})\big| &\leq C\big|W_{\phi\phi}(\zeta^{n+1},\E(\u^{n+1}),\theta^{n+1})\big| + C\big| W_{\phi\phi}(\zeta^{n},\E(\u^{n+1}),\theta^{n+1})\big| \\
 &=:(i)+(ii).
\end{align*}

Before proceeding with the estimates of the second term (ii) on the right-hand side of the above inequality, we note that we have assumed in  \cref{thm:scheme} that $\Th$ is quasi-uniform, that is, we can further resort to inverse inequalities, see \cite[Theorem 4.5.11]{BrennerScott}, of the following form
\begin{align} \label{eq:inverse}
    \|v_h\|_{H^1} \le c_\text{inv} h^{-1} \|v_h\|_{L^2}, \\
    \|v_h\|_{L^p} \le c_\text{inv} h^{d/p-d/q} \|v_h\|_{L^q}, 
\end{align}
which hold for all discrete functions $v_h \in \Vh$ and all $1 \le q \le p \le \infty$. We note that the same estimates hold as well for functions in $\Xh$ and hence in $\Xh^d$.

With the above preparations, we estimate
\begin{align*}
  (i) &\leq \tau\la |W_{\phi\phi}|,(\phi^{n+1})^2 \ra \\ &\leq \tau C\la |\E(\u^{n+1})|^2 + |\zeta^{n+1}|^2 + |\theta^{n+1}|^2 +1,(\phi^{n+1})^2 \ra \\
  &\leq \tau C(\norm{\u^{n+1}}^2_{1,\infty} + \norm{\phi^{n+1}}^2_{0,\infty} + \norm{\theta^{n+1}}^2_{0,\infty})\norm{\phi^{n+1}}^2_0 \\
  &\leq \tau C(\norm{\u^{n+1}}^4_{1,\infty} + \norm{\phi^{n+1}}^4_{0,\infty}+ \norm{\theta^{n+1}}^4_{0,\infty})\norm{\phi^{n+1}}^2_{-1,m} + \frac{\gamma\tau}{4}\norm{\nabla\phi^{n+1}}_0^2\\
&\leq \tau h^{-2d}C(\norm{\u^{n+1}}^4_{1} + \norm{\phi^{n+1}}^4_{0} + \norm{\theta^{n+1}}^4_{0})\norm{\phi^{n+1}}^2_{-1,m} + \frac{\gamma\tau}{4}\norm{\nabla\phi^{n+1}}_0^2 \\
  & \leq \tau h^{-2d}C_1\norm{\phi^{n+1}}^2_{-1,m} + \frac{\gamma\tau}{4}\norm{\phi^{n+1}}_0^2,
\end{align*}
where we used the inverse inequality and the a priori estimate \eqref{eq:apriori} in the last step to bound the term $$\norm{\u^{n+1}}^4_{1}+\norm{\phi^{n+1}}^4_{0}+\norm{\theta^{n+1}}^4_{0}$$ uniformly in $h$ and $\tau$. Using the assumed CFL-like condition $\tau \leq \frac{h^{2d}}{2C_1}$, we find that
\begin{align*}
 (i) \leq  \tfrac{1}{2}\norm{\phi^{n+1}}^2_{-1,m} + \tau\frac{\gamma}{2}\norm{\phi^{n+1}}_0^2.   
\end{align*}
The term $(ii)$ can be estimated in the same manner. Finally, inserting this estimate back into the inequality \eqref{eq:uniqgron}, we obtain
\begin{align*}
 \frac{1}{2}\norm{\phi^{n+1}}_{-1,m}^2 + \tau\big(\tfrac{\gamma}{2}\norm{\nabla\phi^{n+1}}_0^2 &+ \la \Psi_{vex}''(\zeta)\phi^{n+1},\phi^{n+1}\ra \big) \leq 0.
\end{align*}
Hence, it holds $\phi^{n+1}=0$ everywhere. It is straightforward to see that this also implies $\mu^{n+1}=0$ everywhere, which yields a contradiction and implies the uniqueness result, as stated in \cref{thm:scheme}. This finishes the proof.

\section{Numerical illustrations} \label{sec:simulations}

In this section, we illustrate the behaviour of the scheme, see Problem~\ref{prob:split}, by numerical experiments similar to those shown in \cite{fritz2023wellposedness,STORVIK2022}. First, we consider a convergence test showing the first-order convergence in space.  Afterwards, we compare the Cahn-Hilliard-Biot model to the established Cahn-Hilliard-Larch\'e equations when considering typical examples in tumour growth models, see \cite{fritz2023tumor}.
The scheme is implemented in the high-performance multiphysics finite element software NGSolve \cite{schoberl2014c++}. 

The nonlinear system is treated with a Newton method with tolerance $10^{-10}$. We note that it is generally not possible to compute the time averages exactly. Hence, we employ a sufficiently high quadrature rule, that is,
\begin{align*}
g^{av}(\rho) := \frac{1}{\tau}\int_{t^n}^{t^{n+1}} g(\rho(s)) \;\mathrm{d}s \approx \frac{1}{\tau}\sum_{i=1}^{M} \omega_i g(\rho(t_i)).   
\end{align*}
We point out that piecewise linear functions can be parameterised via the new and old time-steps as $\phi_h$ is piecewise linear in time. In the case where $g$ is a polynomial in $\phi_h$, these can be done exactly, which is the case here. As already assumed in \eqref{Ass:Pot} for the analysis, we use the convex-concave splitting of the nonlinear potential $\Psi=\Psi_{cav}+\Psi_{vex}$ into its expansive $\Psi_{cav}$ and contractive part $\Psi_{vex}$, see the initial work \cite{eyre1998} where this splitting was used to prove the unconditional gradient stability of the Cahn-Hilliard equation. In the case of $\Psi(\phi)=\frac14(1-\phi^2)^2$, we set
$\Psi_{vex}'(\phi)=\phi^3$ and $\Psi_{cav}'(\phi)=\phi$.
We treat the expansive part explicitly and the contractive part implicitly. 

The material parameters can be found in Table~\ref{tab:1}. We note that the tensors are written in Voigt notation in two spatial dimensions. The parameter set is taken from \cite{STORVIK2022}. Further, the permeability $\kappa$, compressibility $M$, Biot-Willlis coefficient $\alpha$ and elasticity tensor $\mathbb{C}$ and the consolidation tensor $\mathbb{C}_\nu$ depend on the phase-field through the interpolation function $\pi$ as follows: $$\begin{aligned} \kappa(\phi) &= \kappa_{-1} + \pi(\phi)(\kappa_1-\kappa_{-1}), \\ M(\phi) &= M_{-1} + \pi(\phi)(M_1-M_{-1}), \\ \alpha(\phi) &= \alpha_{-1} + \pi(\phi)(\alpha_1-\alpha_{-1}), \\ \mathbb{C}(\phi) &= \mathbb{C}_{-1} + \pi(\phi)(\mathbb{C}_1 -\mathbb{C}_{-1}), \\ \mathbb{C}_\nu(\phi) &= \mathbb{C}_{\nu,-1} + \pi(\phi)(\mathbb{C}_{\nu,1} -\mathbb{C}_{\nu,-1}).
\end{aligned}$$
Specifically, we choose the interpolation function
\begin{equation*}
    \pi(\phi) = 
    \begin{cases} 
    0,\quad &\phi<-1\\
    \frac14(2+3\phi-\phi^3), \quad &\phi\in [-1,1]\\
    1,\quad &\phi>1
    \end{cases},
\end{equation*}
as originally chosen in \cite{garcke2005CHENumerics} for the simulation of the Cahn-Hilliard-Larch{\'e} equation. In addition, we choose the eigenstrain as $\mathcal{T}(\phi)=\zeta\phi\mathbf{I}$ with $\zeta$ given in Table~\ref{tab:1}.

\begin{table}[htbp!]
\centering
\caption{Table of simulation parameters.}
\begin{tabular}{c|c|c|c|c}
 Symbol & Value & & Symbol & Value\\
\hline
 $m$ & 1 & &  $\alpha_{-1}$, $\alpha_1$ & 1, 0.5 \\
 $\gamma$ & $10^{-4}$ & & $\kappa_{-1}$, $\kappa_{1}$& 1, 0.1 \\
$\xi$ & 0.3  & &  $M_{-1}, M_1$ & 1, 0.1\\
 $\CC_{-1}$ &$\begin{pmatrix}  4 & 2 & 0 \\  2 & 4 & 0 \\  0 & 0 & 8\end{pmatrix}$ & &
 $\CC_{1}$ &$\begin{pmatrix}  1 & 0.5 & 0 \\  0.5 & 1 & 0 \\ 0 & 0 & 2\end{pmatrix} $ \\
 & & & & \\
 $\CC_{\nu,-1}$ &$\begin{pmatrix}  1 & 0.5 & 0 \\  0.5 & 1 & 0 \\ 0 & 0 & 2\end{pmatrix}$ & &
$\CC_{\nu,1}$  &$\begin{pmatrix}  1 & 0.5 & 0 \\  0.5 & 1 & 0 \\ 0 & 0 & 2\end{pmatrix} $
\end{tabular}
\label{tab:1}
\end{table}

\subsection{Convergence test}

We consider the error in space at the final time $T=0.01$ with a respective time-step size of $\tau = 10^{-5}$ by comparing the numerical solutions $(\phi_{h},\mu_{h},\u_{h},\theta_{h},p_{h})$ with those computed on uniformly refined grids, $(\phi_{h/2},\mu_{h/2},\u_{h/2},\theta_{h/2},p_{h/2})$ as no analytical solution to the Cahn-Hilliard-Biot system is available.
The error quantities for the fully discrete scheme are defined by
\begin{align*}
	\hspace{-1em}e_{h} &= \norm*{\phi_{h} - \phi_{h/2}}_{H^1}^2  + \norm*{\E(\u_{h}) - \E(\u_{h/2})}_{L^2}^2  + \norm*{\theta_{h} - \theta_{h/2}}_{L^2}^2\\
    &+  \norm*{\mu_{h} - \mu_{h/2}}_{H^1}^2 + \norm*{p_{h} - p_{h/2}}_{H^1}^2
\end{align*}
as well as the separated errors $e^\phi_{h}, e^\mu_{h}, e^{\u}_{h}$ and $e^{p}_{h}$, which denote the related single quantities from the second error norm, e.g.,  $e^\phi_{h}=\norm*{\phi_{h} - \phi_{h/2}}_{H^1}^2$.

\begin{experiment}
We assume a smooth set of initial data given by
\begin{align*}
 \phi_0 &=-0.1 + 0.01\sin(2\pi x)\sin(2\pi y),\quad \u_0=\mathbf{0}, \quad\theta_0=0, 
\end{align*}
and the source function $r=s=0$, $\mathbf{f}=\mathbf{0}$.
\end{experiment}

For the spatial step sizes $h_k=2^{-k}$ with $k\in \{2,\ldots,6\}$ of the domain $\Omega=(0,1)^2$ we obtain the error rates as illustrated in Table \ref{tab:err}. We observe second-order convergence in the squared norms, which implies the expected first-order convergence in space.
\begin{table}[htbp!]
	\centering
	\small
	\caption{Errors and experimental orders of convergence (eoc).} 
 \label{tab:err}
	\begin{tabular}{c||c|c|c|c|c|c|c|c|c|c}
		% \hline
		$ k $ & $e_{h}$ & eoc & $e^\phi_{h}$  & eoc & $e^\mu_{h}$ & eoc & $e^\u_{h}$ & eoc & $e^p_{h}$ & eoc  \\
		\hline
		$ 2 $ & $  1.23\cdot 10^{-3}$ & --  & $1.20\cdot10^{-3}$ & --  & $3.72\cdot 10^{-5}$ & --  & $6.80\cdot 10^{-8}$  & --  & $1.01\cdot 10^{-6}$  &   --    \\
		$ 3 $ & $  6.27\cdot 10^{-4}$ & 0.98 & $6.02\cdot10^{-4}$ & 0.99 & $2.19\cdot 10^{-5}$ & 0.77 & $3.05\cdot 10^{-8}$  & 1.16 & $5.23\cdot 10^{-7}$  &   0.95   \\
		$ 4 $ & $  2.67\cdot 10^{-4}$ & 1.23 & $2.57\cdot10^{-4}$ & 1.23 & $7.13\cdot 10^{-6}$ & 1.62 & $7.86\cdot 10^{-9}$  & 1.96 & $1.50\cdot 10^{-7}$  &   1.80   \\
		$ 5 $ & $  6.64\cdot 10^{-5}$ & 2.00 & $6.36\cdot10^{-5}$ & 2.02 & $1.79\cdot 10^{-6}$ & 1.99 & $2.01\cdot 10^{-9}$  & 1.97 & $3.75\cdot 10^{-8}$  &   2.00  \\
        $ 6 $ & $  1.60\cdot 10^{-5}$ & 2.05 & $1.51\cdot10^{-5}$ & 2.07 & $4.39\cdot 10^{-7}$ & 2.03 & $5.07\cdot 10^{-10}$ & 1.99 & $9.26\cdot 10^{-9}$  &   2.02   
	\end{tabular}
\end{table}

\subsection{Validation experiments}

\begin{experiment} \label{exp:lshape}
We consider an initial condition with three bubbles of phase $\phi=1$ while the rest of the domain $\Omega=(0,1)^2$ is given by $\phi=-1$ together with zero initial displacement and zero volumetric fluid content. Thus, we assume
\begin{align*}
 \phi_0 &= 2 - \tanh\left(\frac{\mathcal{B}(0.3,0.3,0.15)}{0.005}\right) - \tanh\left(\frac{\mathcal{B}(0.3,0.7,0.15)}{0.005}\right)  - \tanh\left(\frac{\mathcal{B}(0.7,0.3,0.15)}{0.005}\right), \\
 \u_0&=\mathbf{0}, \qquad \theta_0=0, \qquad \mathcal{B}(x_0,y_0,r):= (x-x_0)^2+ (y-y_0)^2 - r^2
\end{align*}
with the source functions $r=s=0$, $\mathbf{f}=\mathbf{0}$.
\end{experiment}

We depict the result in Figure \ref{pic:Lshape}. One can clearly see that after a short time period at $t\approx 0.02$, the three bubbles merge and begin to agglomerate. This agglomeration process contrasts with the usual Cahn-Hilliard dynamics subjected to elastic contribution and results in the L shape, which expands over time, that is, at $t\approx 0.06$ and finally at $t\approx 2$. Note that this specific L shape is due to the special choice of the eigenstrain $\mathcal{T}=\zeta\phi\mathbf{I}$, which implies that the stress-free state is located at $\phi=0$, that is, at the interface. Hence, the system tries to expose as much interface as possible given the contrary Cahn-Hilliard dynamics.

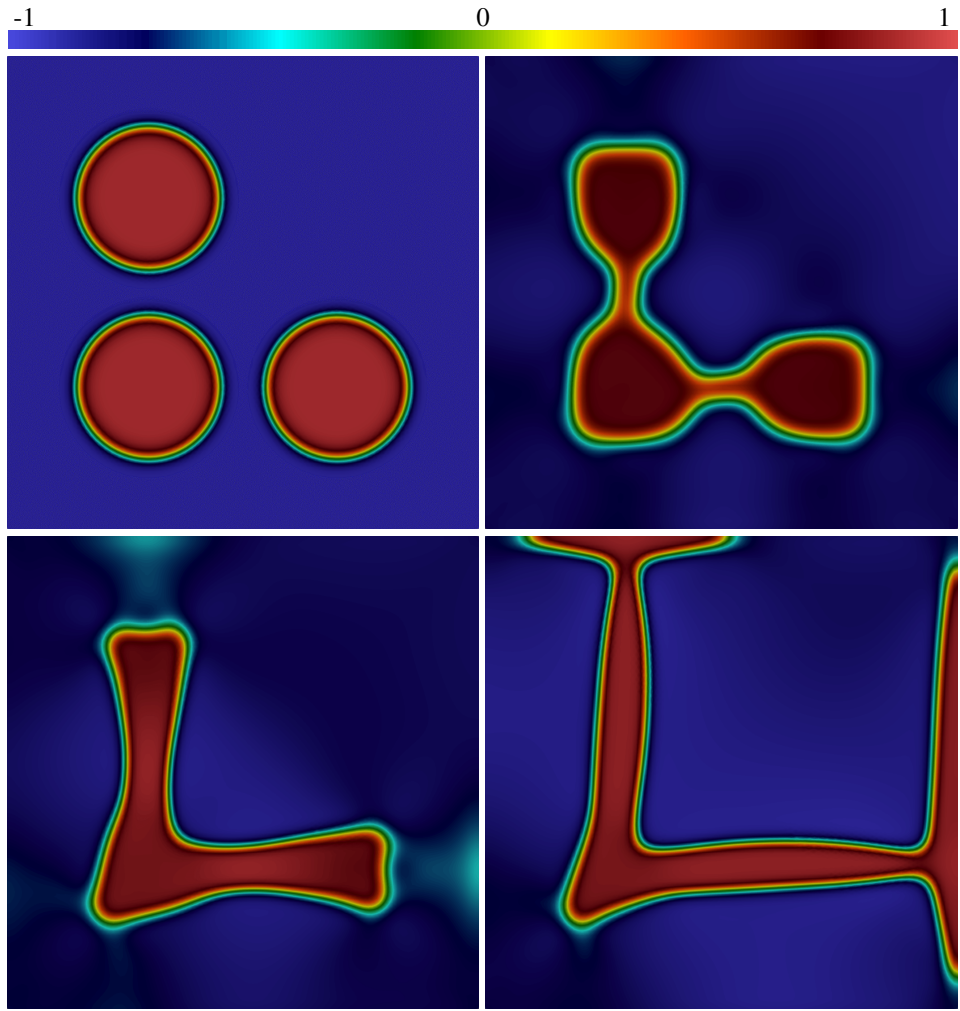
\begin{figure}[htbp!]
\centering
\begin{tikzpicture}
    \draw (0.03, 0) node[inner sep=0] {\includegraphics[width=.705\textwidth]{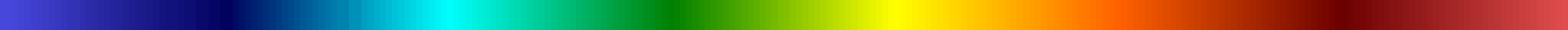}};
    \draw (-6, 0.3) node {-1};
    \draw (0.03, 0.3) node {0};
    \draw (6.1, 0.3) node {1};
\end{tikzpicture} \\[.1cm]\includegraphics[width=.35\textwidth]{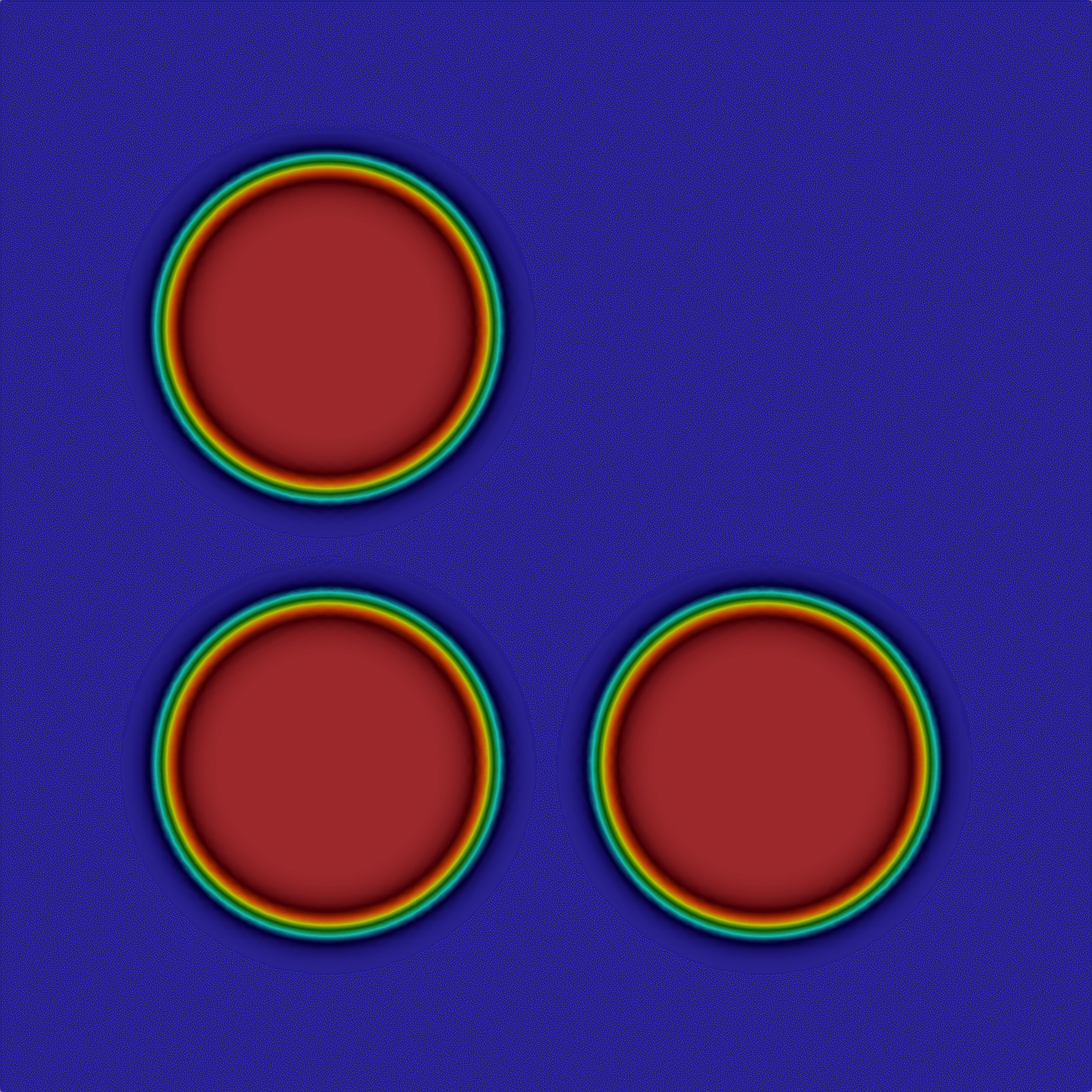} \includegraphics[width=.35\textwidth]{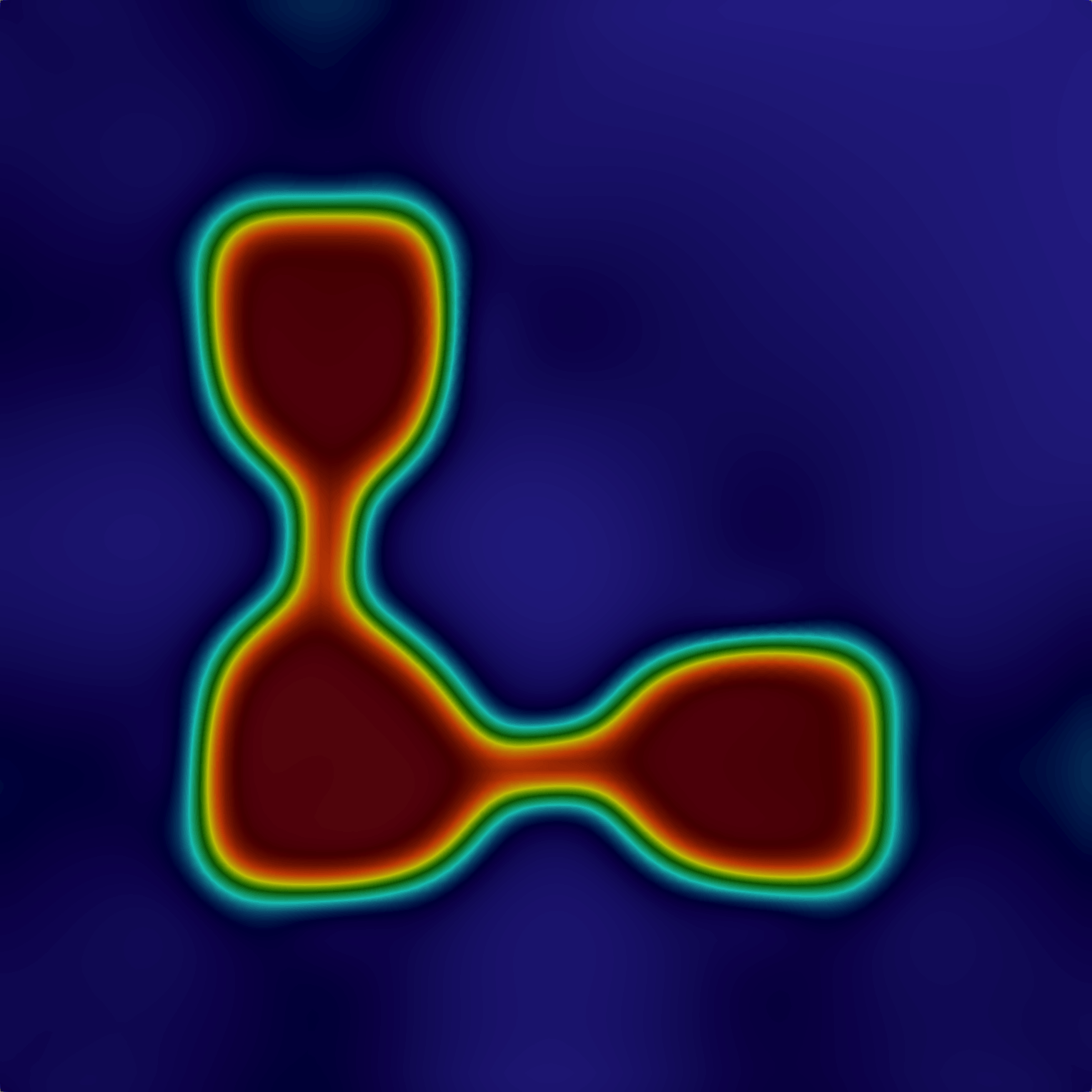} \\[.1cm]
\includegraphics[width=.35\textwidth]{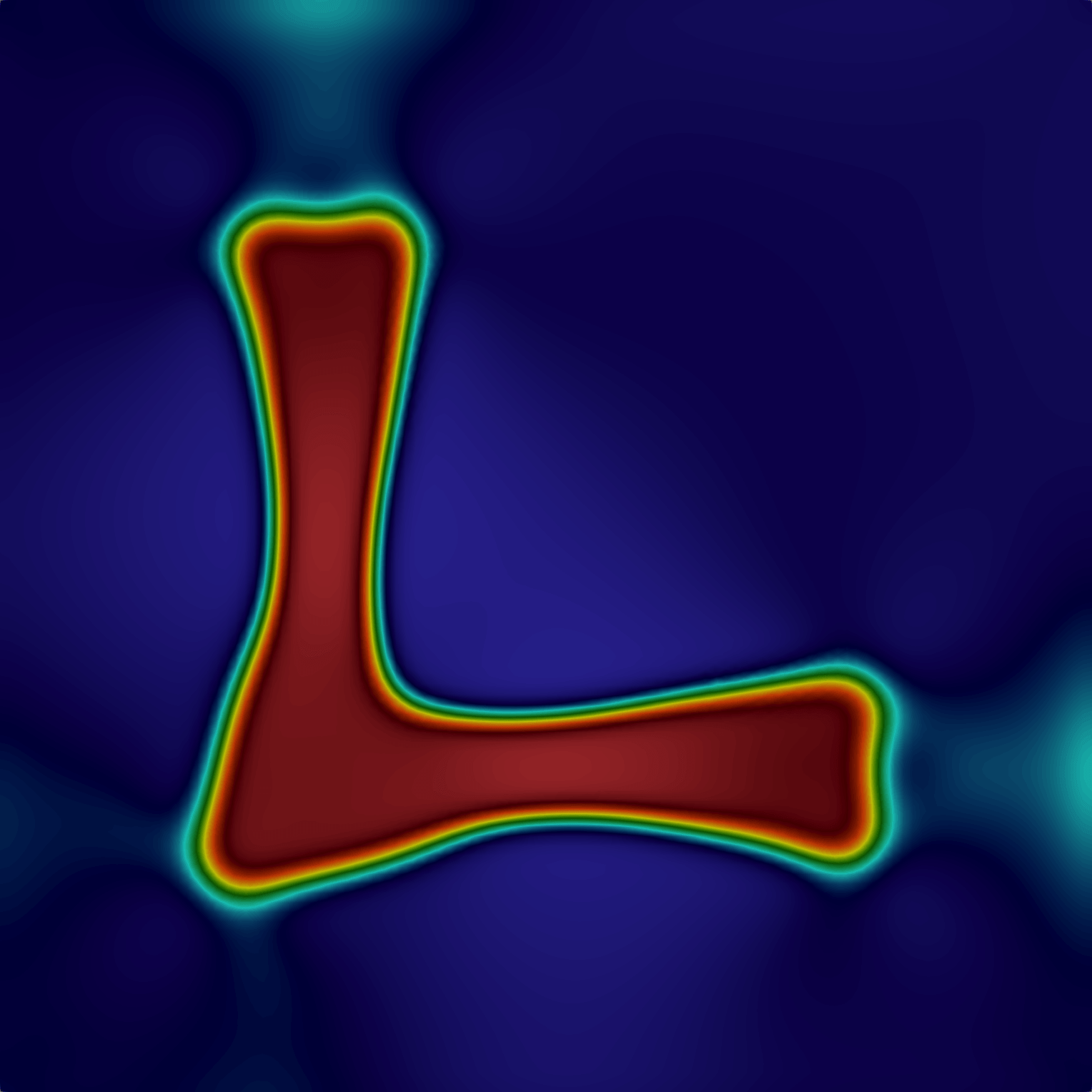} \includegraphics[width=.35\textwidth]{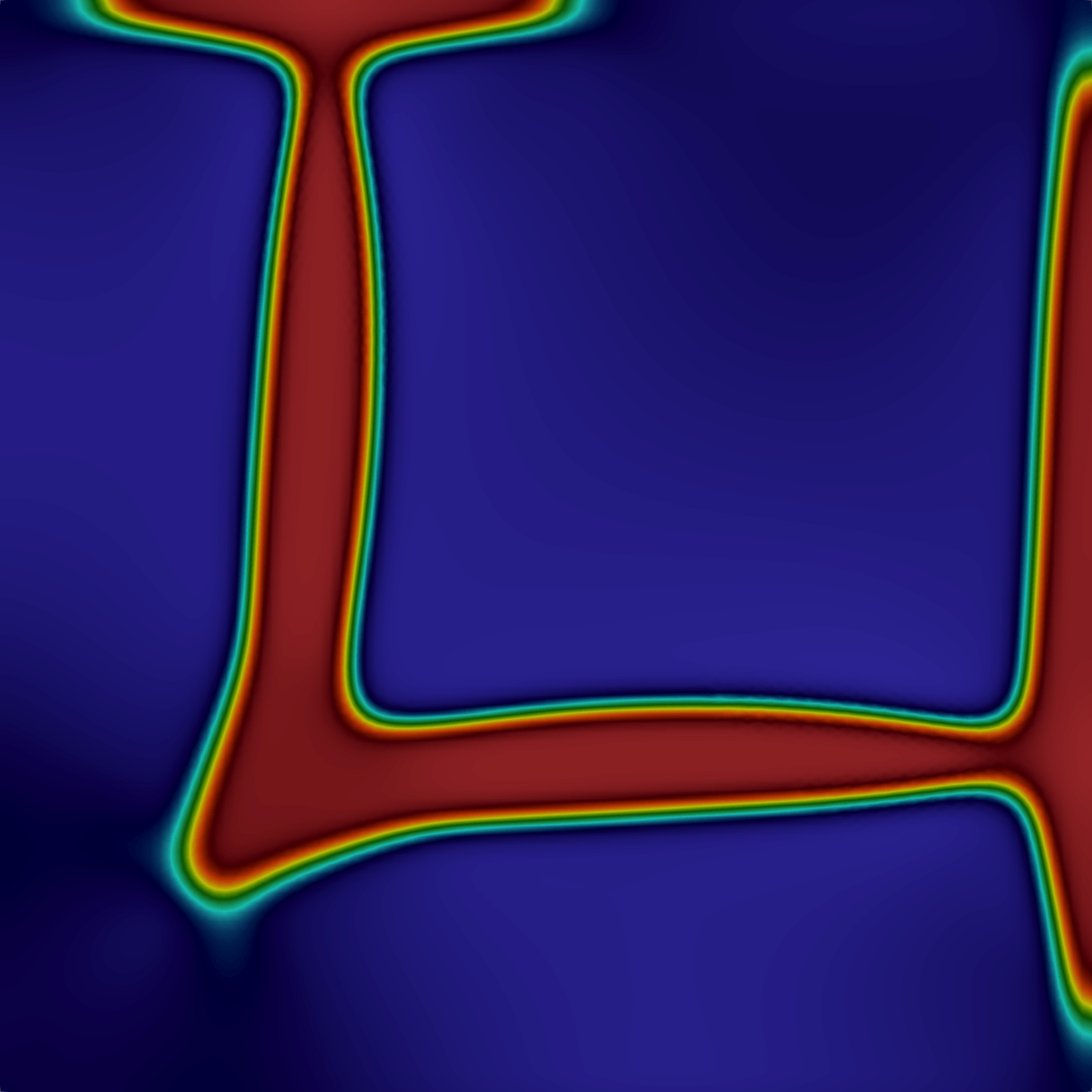} 
\caption{Experiment \ref{exp:lshape}: Snapshots of the phase field $\phi$ at the times $t\in\{0,0.02,0.06,2\}$ with mesh size $h_{\max} \approx 10^{-2}$ and time-step $\tau=10^{-3}$ (from top left to bottom right). \label{pic:Lshape}}
\end{figure}

The next experiment is related to tumour growth. In this case, we compare the Cahn-Hilliard-Biot system (CHB) with the Cahn-Hilliard-Larch{\'e} system (CHL), which is achieved by assuming compressibility $M=0$. We recall that $\phi$ describes the difference in volume fractions, that is, $\{\phi=1\}$ represents unmixed tumour tissue, while $\{\phi=-1\}$ represents surrounding healthy tissue. 
\begin{experiment}\label{exp:tumor}
We consider a single bubble (here the tumour) parametrized by
\begin{align*}
 \phi_0 &= - \tanh\left(\frac{\mathcal{B}(0.5,0.5,0.15)}{0.005}\right), \qquad\u_0=\mathbf{0}, \qquad \theta_0=0.    
\end{align*}
Furthermore, we choose the logistic growth function $r(\phi)=\frac{5}{2}(1-\phi^2)$, $s=\mathbf{f}=0$ and mobility $m(\phi)=10^{-14} + \frac{1}{16}(\phi^2-1)^2.$ Additionally, the following parameters are changed in comparison to Table \ref{tab:1}
$$\begin{aligned}
 \mathbb{C}_{\nu} = 0, \quad \mathcal{T}(\phi)= \frac{1}{2}\zeta(\phi+1)\I, \quad \CC_{-1}=\begin{pmatrix}  6 & 4 & 0 \\  4 & 6 & 0 \\  0 & 0 & 1\end{pmatrix}, \quad \CC_{1}=\begin{pmatrix}  1.55 & 0.38 & 0 \\  0.38 & 1.55 & 0 \\  0 & 0 & 0.58\end{pmatrix}.
\end{aligned}$$
\end{experiment}

Snapshots of these experiments are given in Figures \ref{pic:tumorCHB}--\ref{pic:tumorDiff}. In Figure \ref{pic:tumorCHB} we show the evolution of the CHB system, which appears to be very similar to the evolution of the CHL system in Figure \ref{pic:tumorCHL}. To highlight the difference, we also compute the difference between both solutions, that is, $\snorm{\phi_{\text{CHB}}-\phi_{\text{CHL}}}$ and plot it in Figure~\ref{pic:tumorDiff}.  From this, we see that both models agree very well apart from the interface, while on the interface we observe errors up to the order $10^{-3}$. Since both systems start from the same data and the elastic/poro-elastic system is initialised with zero and $s=0$, this suggests that the CHB system exhibits a distinct feature at the interface.  This new interface behaviour warrants further detailed investigation.

\begin{figure}[htbp!]
\centering
\begin{tikzpicture}
    \draw (0.03, 0) node[inner sep=0] {\includegraphics[width=.705\textwidth]{pics/colornew3.png}};
    \draw (-6, 0.3) node {-1};
    \draw (0.03, 0.3) node {0};
    \draw (6.1, 0.3) node {1};
\end{tikzpicture} \\[.1cm]
\includegraphics[width=.35\textwidth]{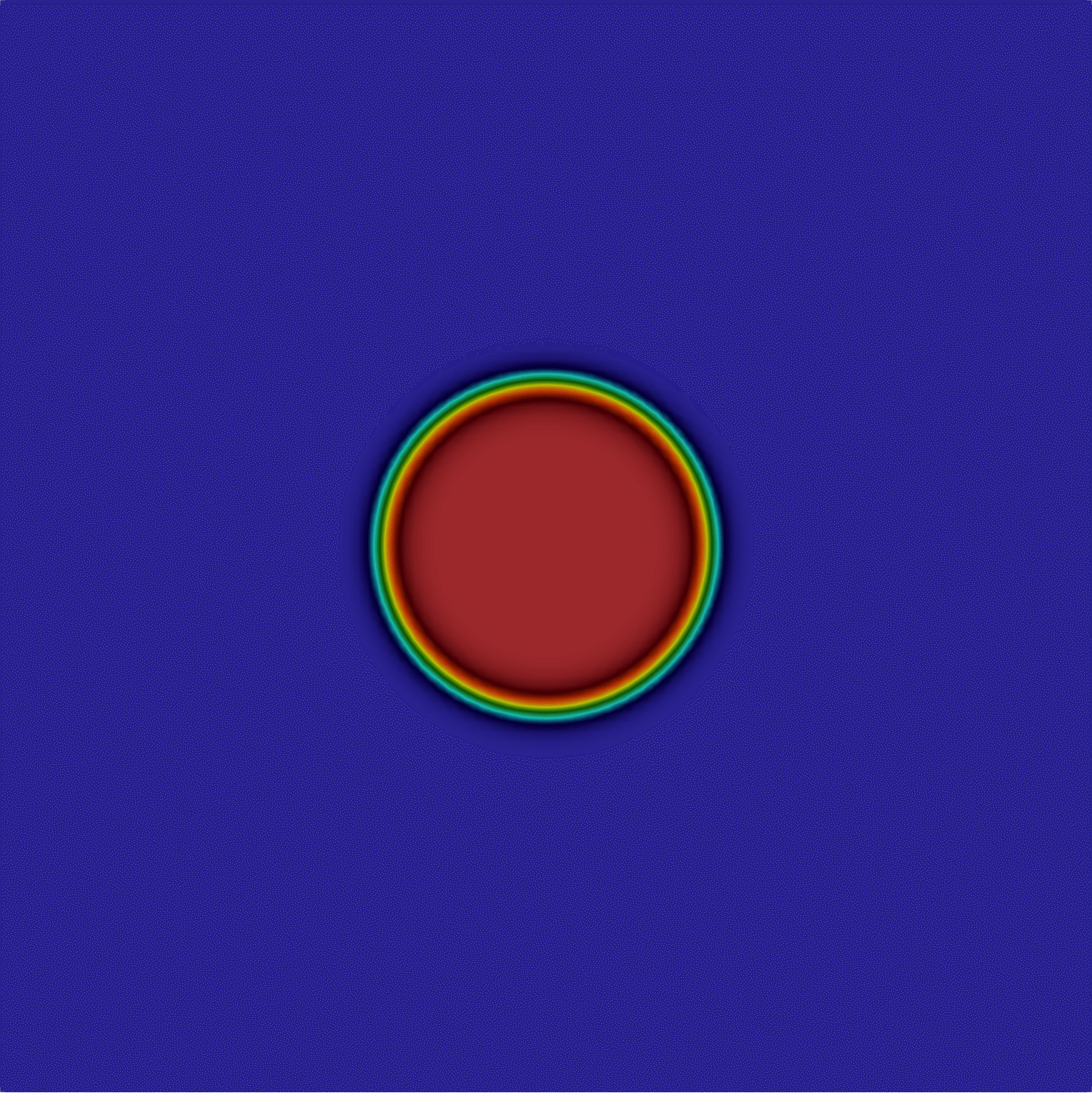}
\includegraphics[width=.35\textwidth]{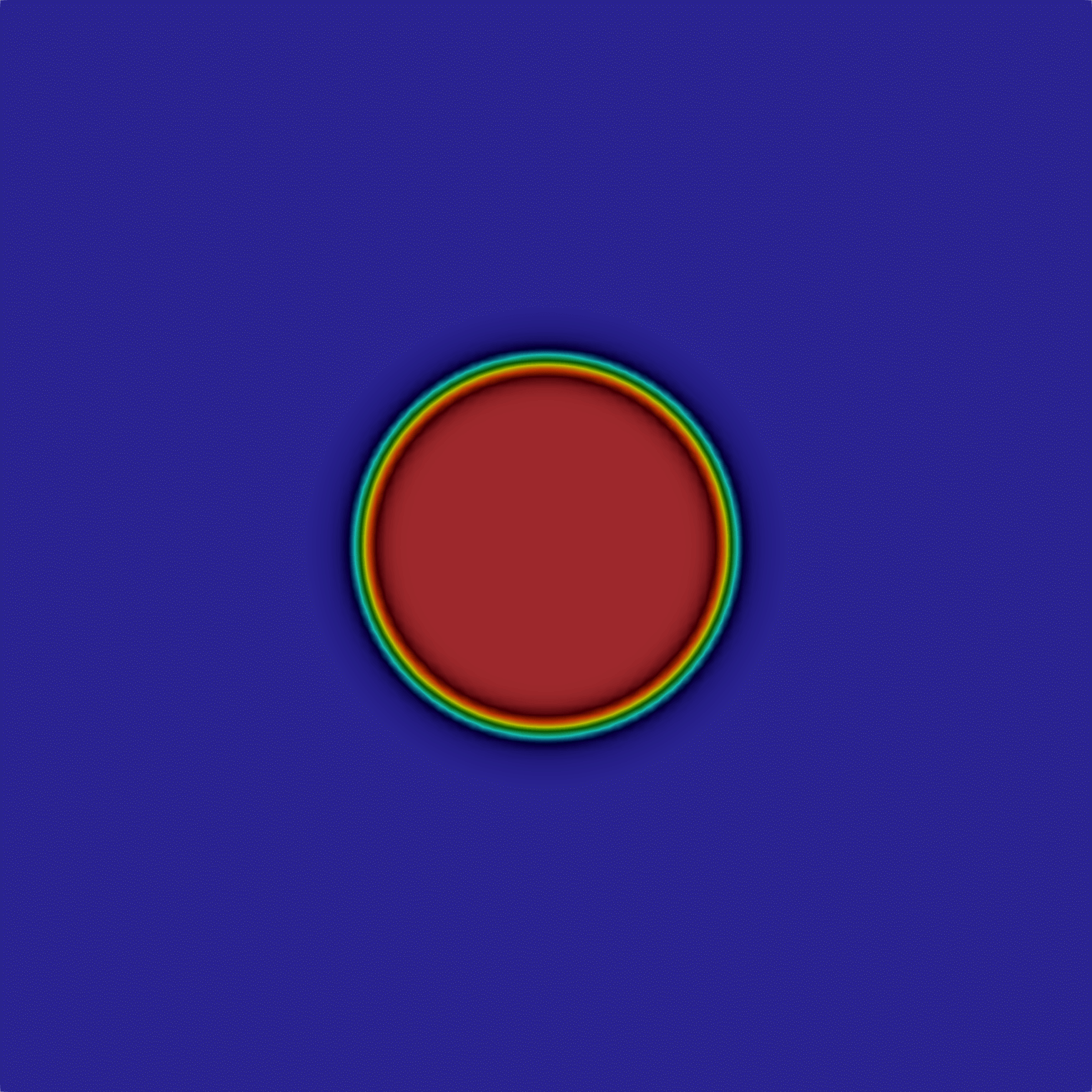}\\[.1cm]
\includegraphics[width=.35\textwidth]{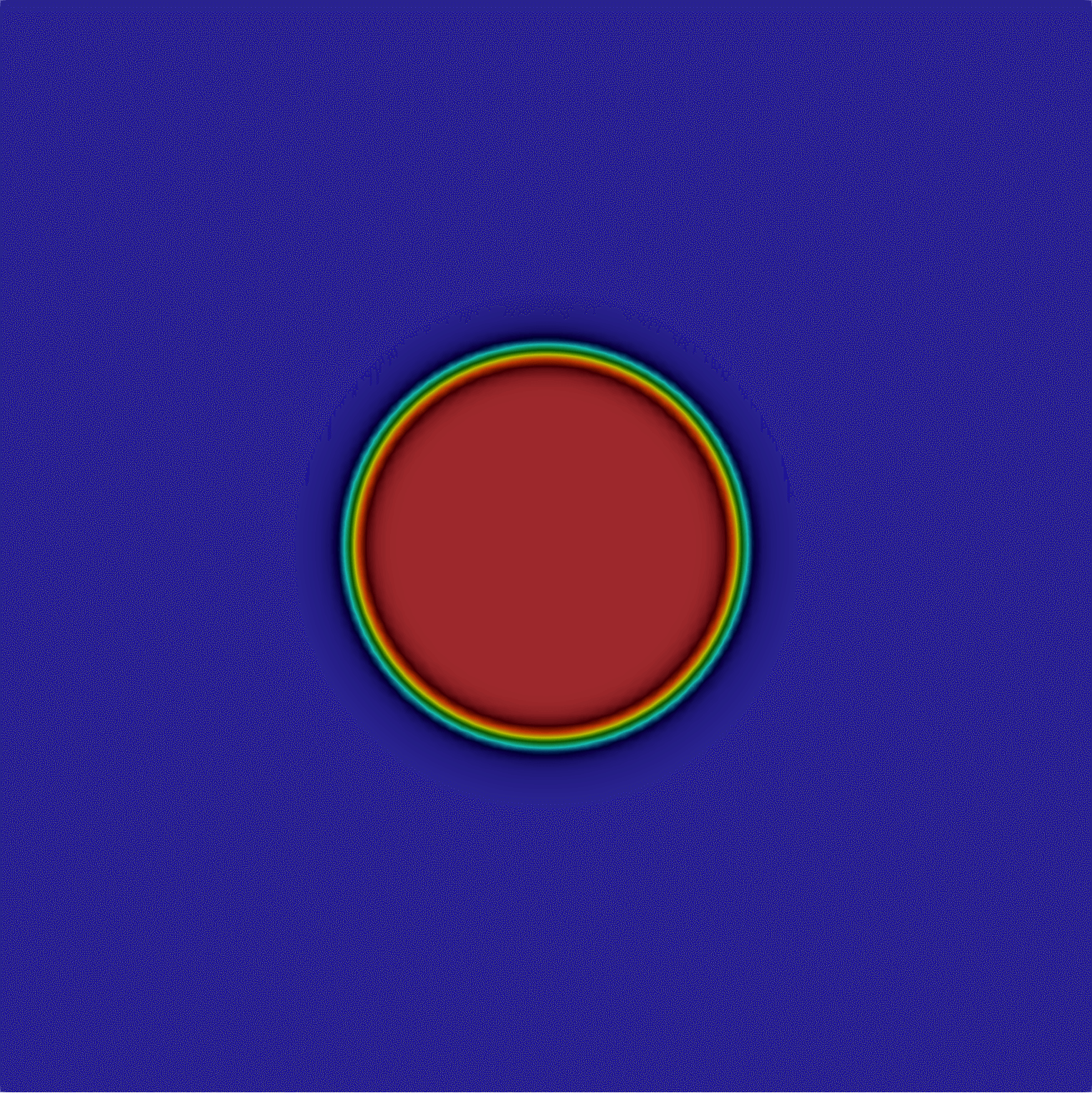}
\includegraphics[width=.35\textwidth]{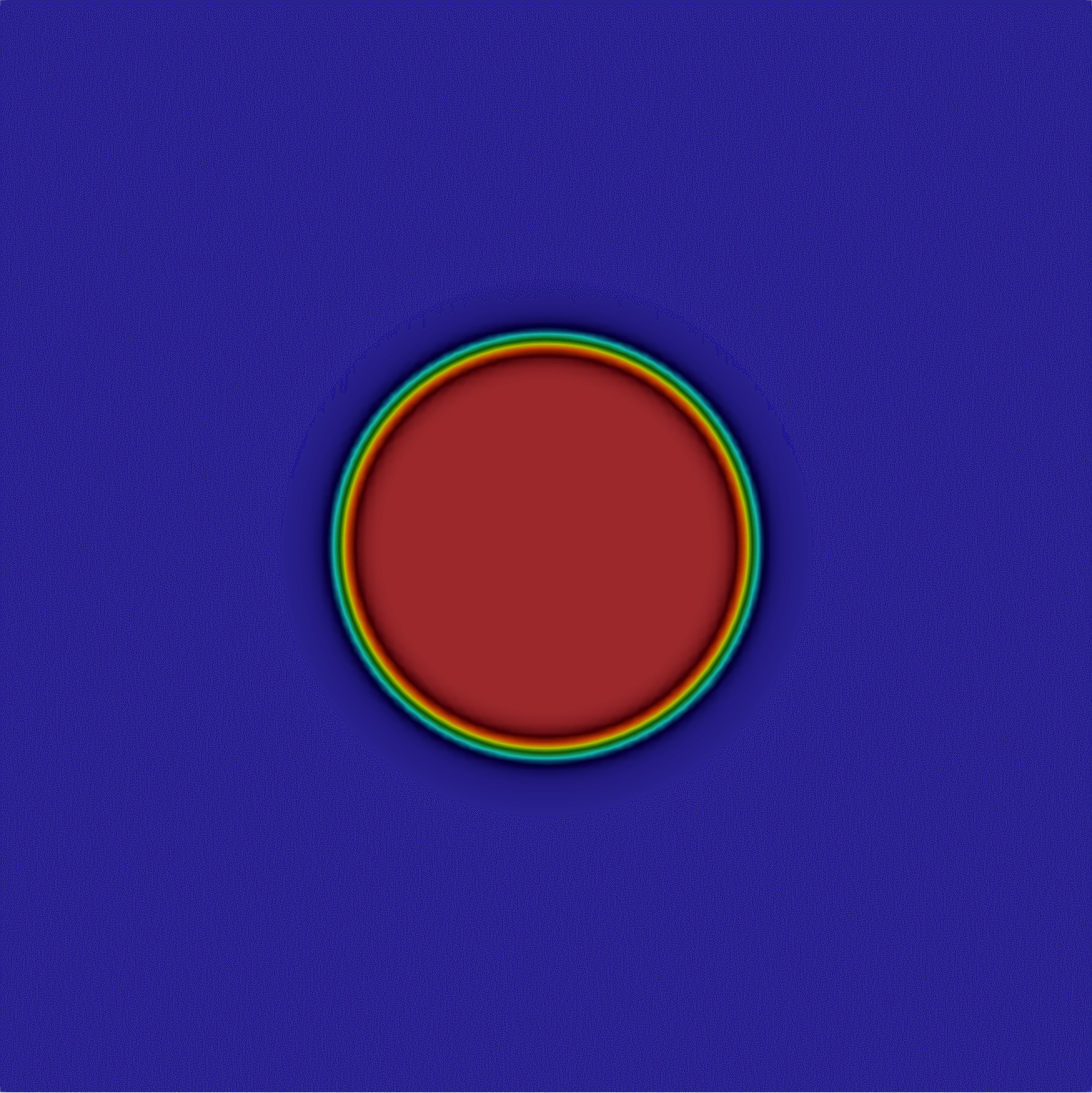}
\caption{Experiment \ref{exp:tumor} for the CHB system: Snapshots of the phase field $\phi$ at the times $t\in\{0,0.5,0.75,1\}$  with mesh size $h_{\max} \approx 10^{-2}$ and time-step $\tau=10^{-3}$ (from top left to bottom right).  \label{pic:tumorCHB}}
\end{figure}

\begin{figure}[htbp!]
\centering
\begin{tikzpicture}
    \draw (0.03, 0) node[inner sep=0] {\includegraphics[width=.705\textwidth]{pics/colornew3.png}};
    \draw (-6, 0.3) node {-1};
    \draw (0.03, 0.3) node {0};
    \draw (6.1, 0.3) node {1};
\end{tikzpicture} \\[.1cm]\includegraphics[width=.35\textwidth]{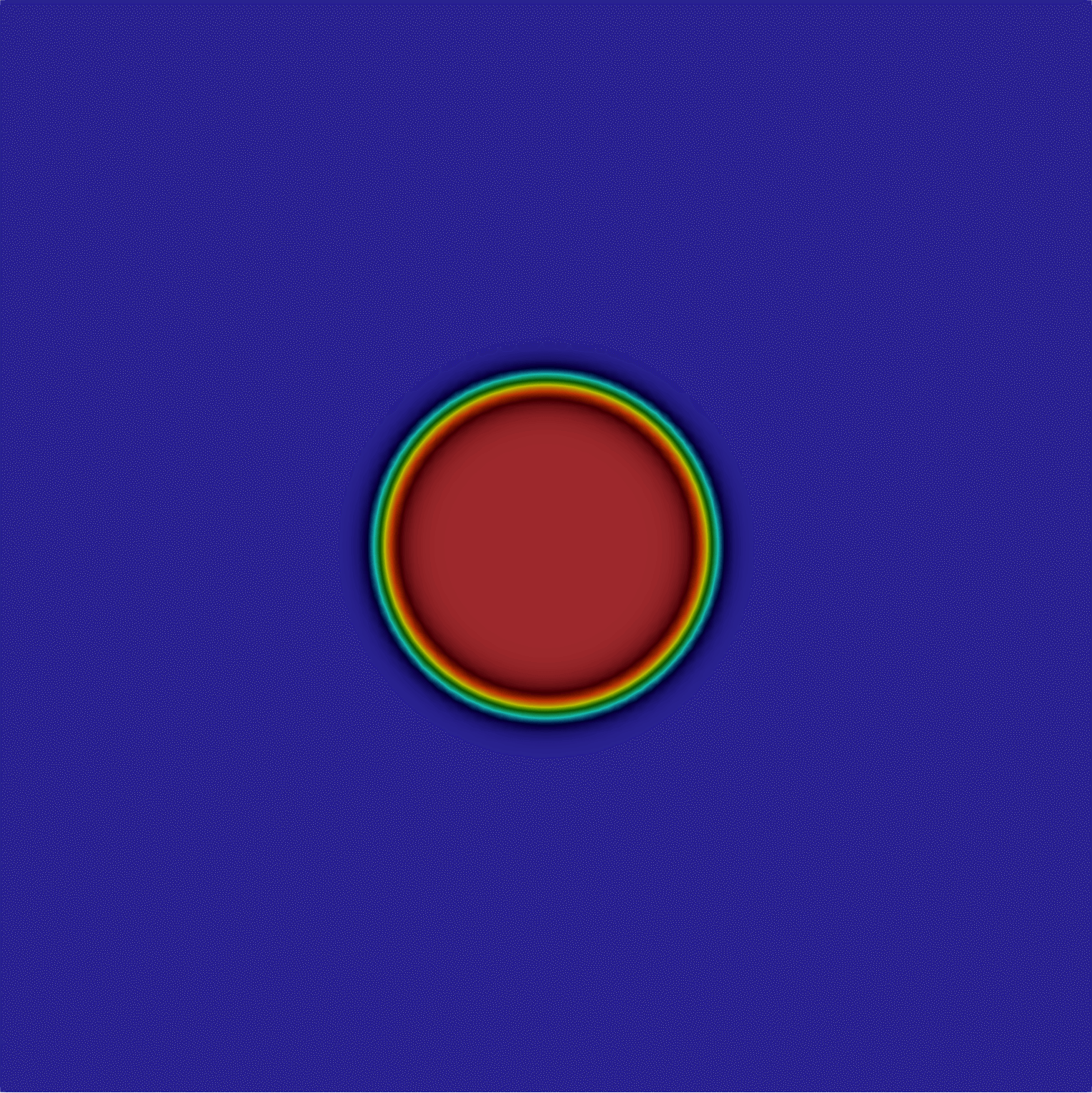} \includegraphics[width=.35\textwidth]{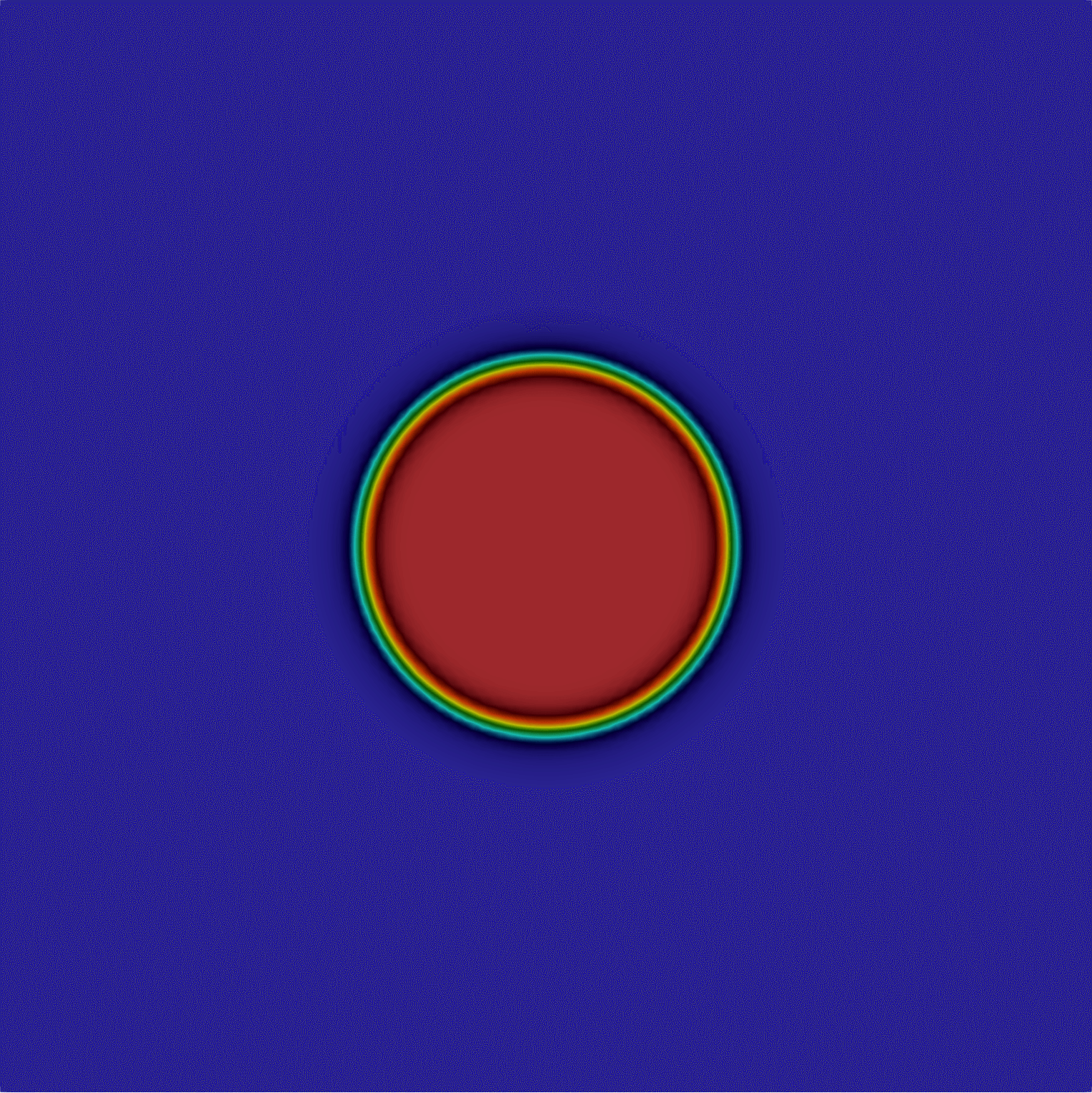}\\[.1cm]
\includegraphics[width=.35\textwidth]{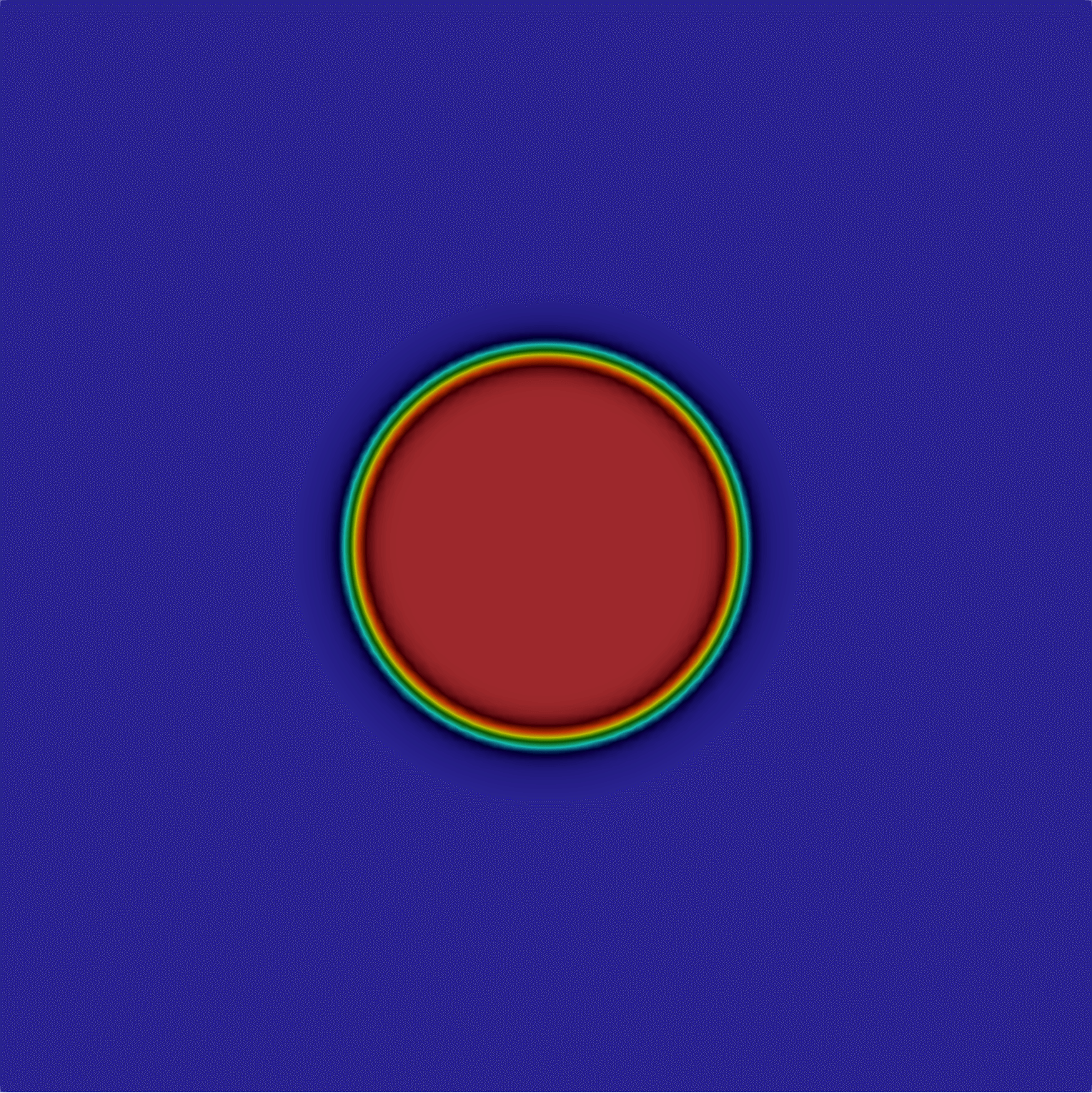} \includegraphics[width=.35\textwidth]{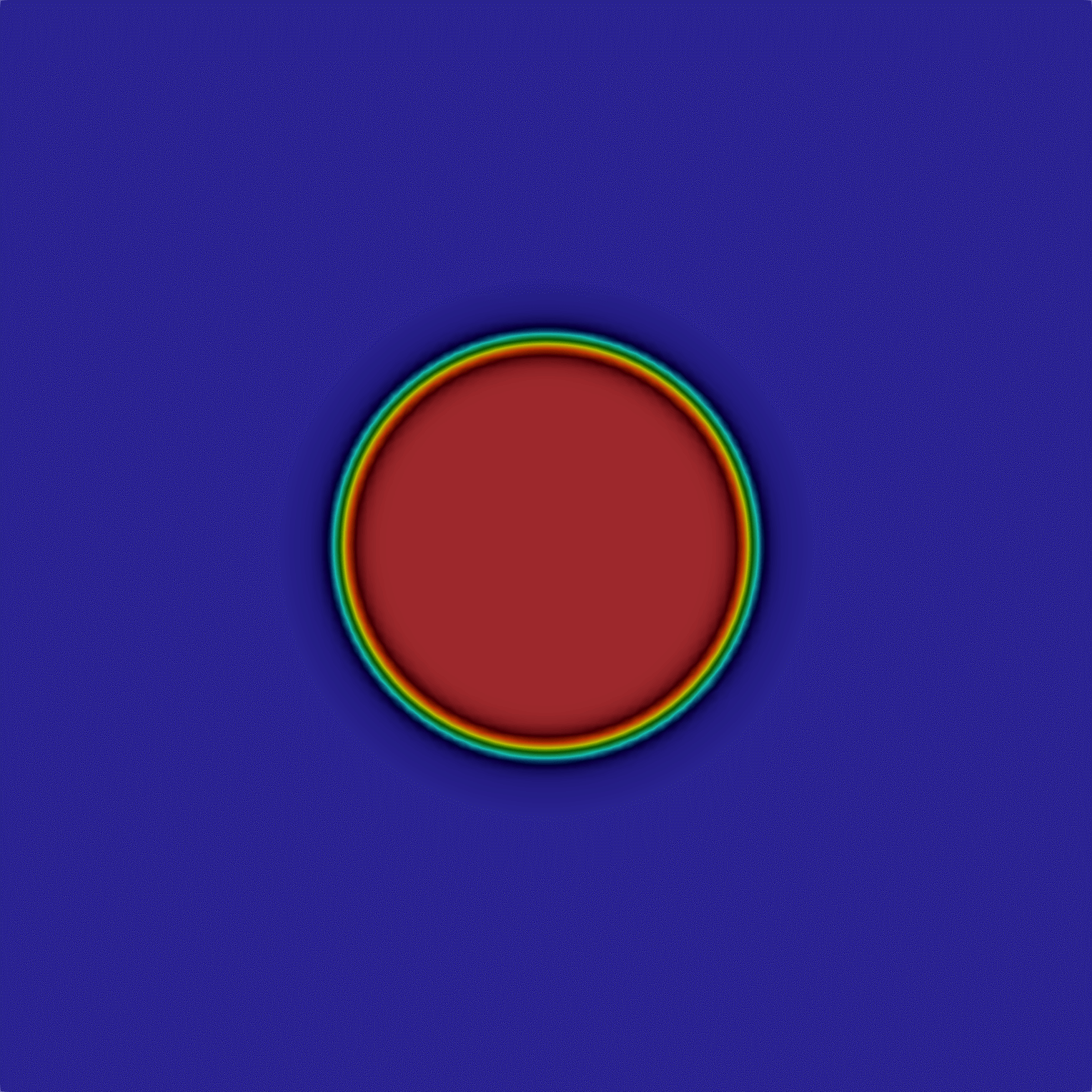}
\caption{Experiment \ref{exp:tumor} for the CHL system: Snapshots of the phase field $\phi$ at times $t\in\{0,0.5,0.75,1\}$ with mesh size $h_{\max} \approx 10^{-2}$ and time-step $\tau=10^{-3}$ (from top left to bottom right). \label{pic:tumorCHL}}
\end{figure}

\begin{figure}[htbp!]
\centering
\begin{tikzpicture}
    \draw (0.03, 0) node[inner sep=0] {\includegraphics[width=.79\textwidth]{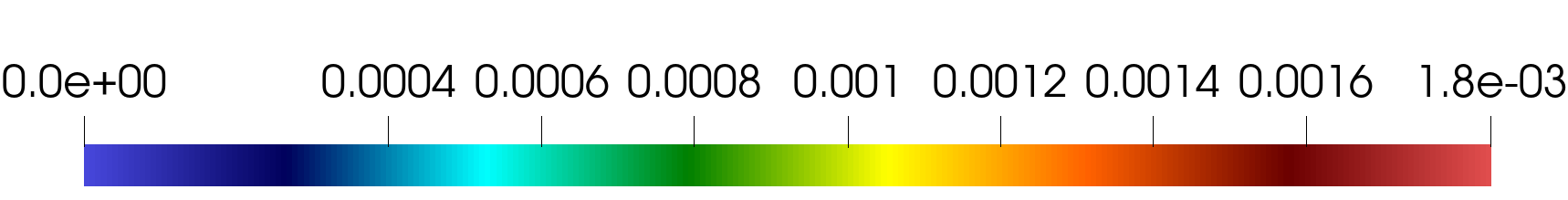}};
    %\draw (-6, 0.3) node {-1};
    %\draw (0.03, 0.3) node {0};
    %\draw (6.1, 0.3) node {1};
\end{tikzpicture} \\[.1cm]
\includegraphics[width=.35\textwidth]{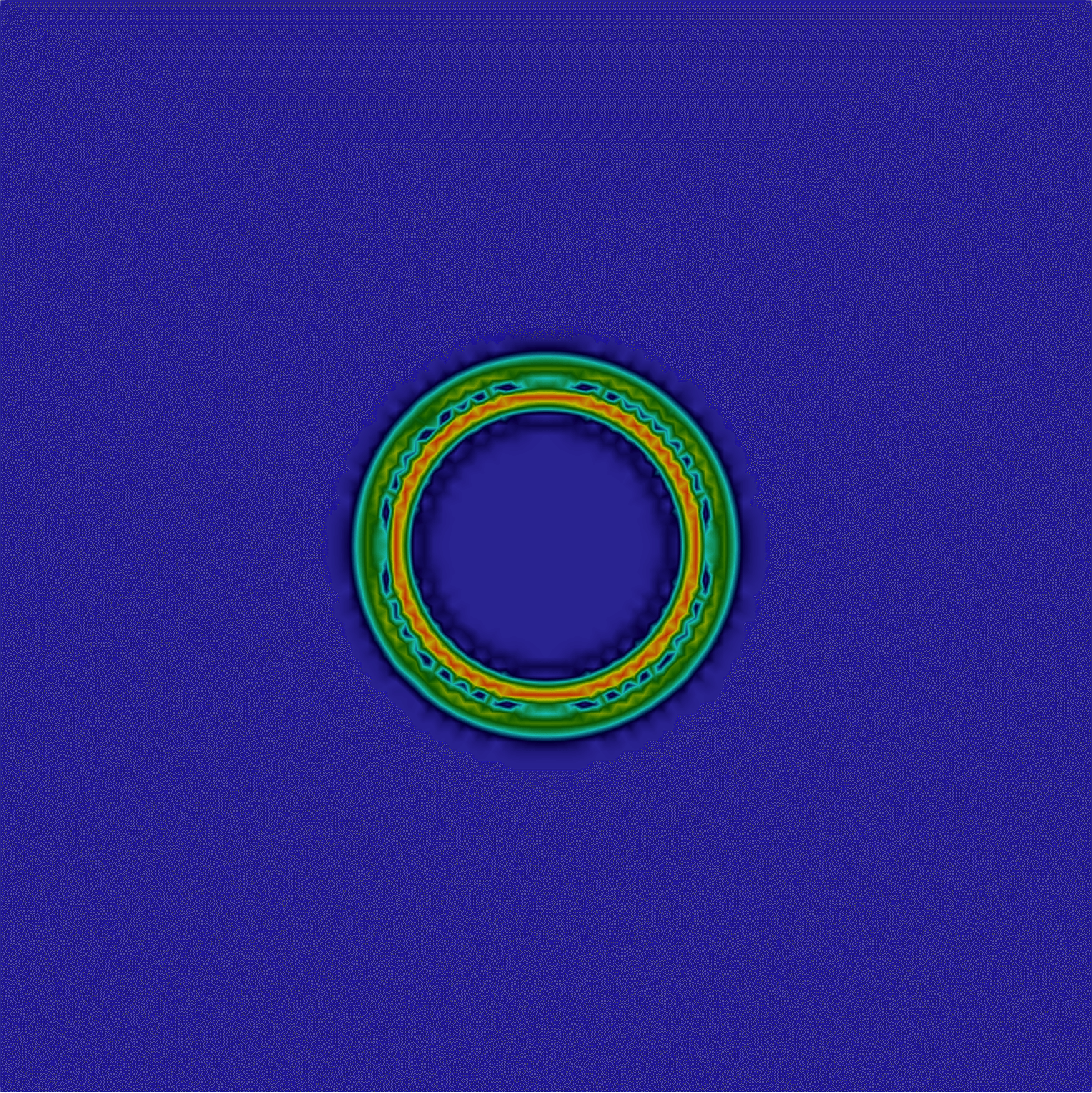}
\includegraphics[width=.35\textwidth]{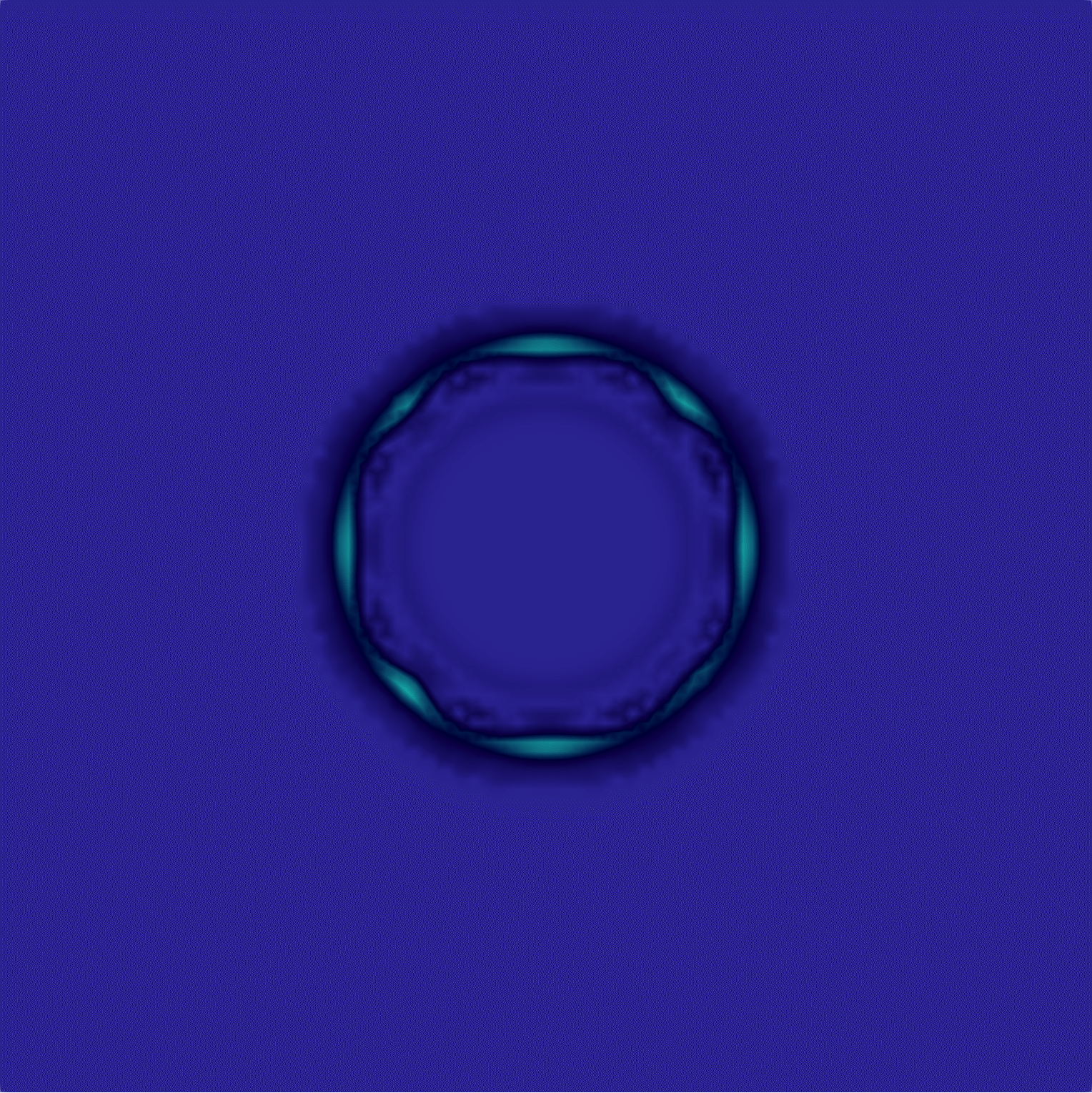}\\[.1cm]
\includegraphics[width=.35\textwidth]{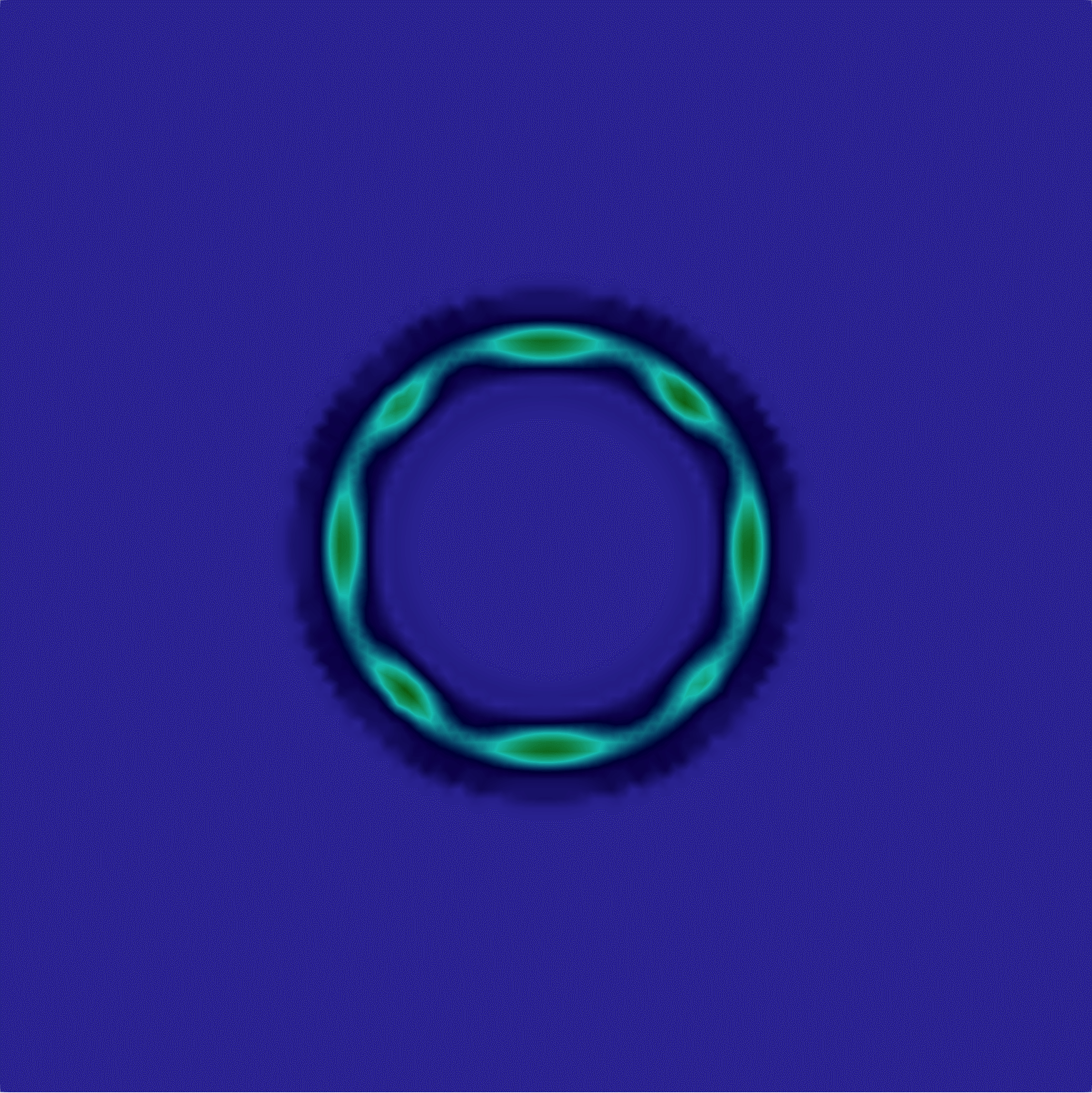}
\includegraphics[width=.35\textwidth]{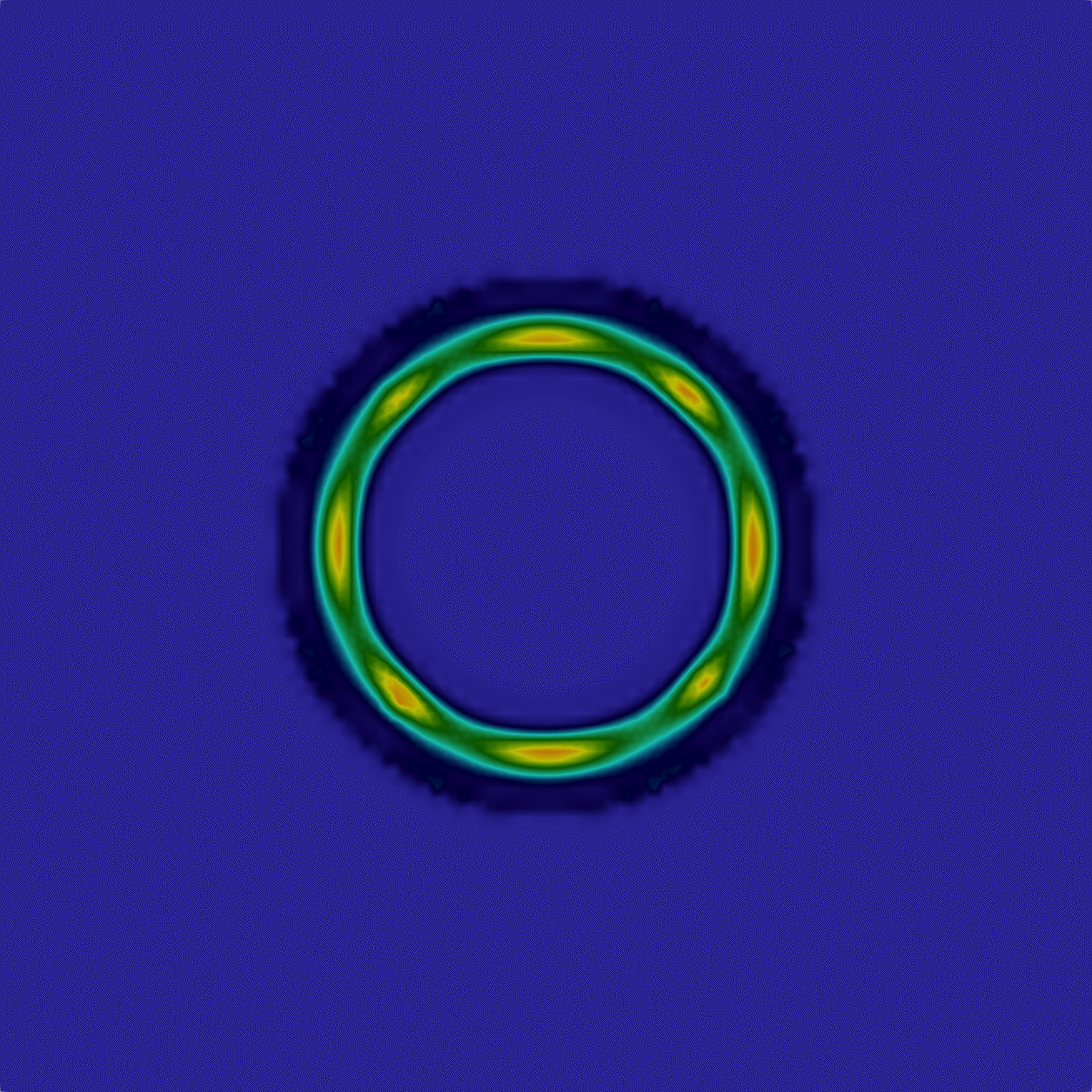}
\caption{Experiment \ref{exp:tumor}: Difference between the phase-fields of the CHB and CHL solutions in the times $t\in\{0.01,0.5,0.75,1\}$ (from top left to bottom right). \label{pic:tumorDiff}}
\end{figure}

The final example combines both experiments, that is, we consider tumour growth using 3 initial tumours in three space dimensions by choosing the domain $\Omega=(0,1)^3$.
\begin{experiment}
We consider the following set of initial data
\begin{align*}
 \phi_0 &= 2 - \tanh\left(\frac{\mathcal{B}(0.3,0.3,0.3,0.15)}{0.01}\right) - \tanh\left(\frac{\mathcal{B}(0.3,0.7,0.3,0.15)}{0.01}\right) \\
 &\phantom{= 2}- \tanh\left(\frac{\mathcal{B}(0.7,0.3,0.3,0.15)}{0.01}\right), \qquad\u_0=\mathbf{0}, \qquad \theta_0=0\\ 
 \mathcal{B}(x_0,y_0,z_0,r)&:= (x-x_0)^2+ (y-y_0)^2+ (z-z_0)^2 - r^2.    
\end{align*}
Furthermore, we choose $r(\phi)=\frac{5}{2}(1-\phi^2)$, $s=\mathbf{f}=0$ and $m(\phi)=10^{-14} + \frac{1}{16}(\phi^2-1)^2.$  Additionally the following parameters are changed in comparison to Table \ref{tab:1}
\begin{align*}
 \CC_{-1}=&\begin{pmatrix}  6 & 4 & 0 \\  4 & 6 & 0 \\  0 & 0 & 1\end{pmatrix}, \quad \CC_{1}=\begin{pmatrix}  1.55 & 0.38 & 0 \\  0.38 & 1.55 & 0 \\  0 & 0 & 0.58\end{pmatrix}, \quad \mathbb{C}_{\nu} = 10^{-2}\CC_{1}, \quad \mathcal{T}= \frac{1}{2}\zeta(\phi+1)\I,
\end{align*}
\end{experiment}

The simulation results are shown in Figure \ref{pic:3d}. The three initial bubbles start to grow over time and start to connect around $t\approx 0.8$. Afterwards, the agglomeration process starts, which can be seen at $t\approx 1$. Note that, as discussed in Experiment \ref{exp:lshape}, due to the different choice of the eigenstrain $\mathcal{T}$ we do not observe the L shape in this simulation. Instead, the tumour keeps its circular shape.

\begin{comment}
\begin{figure}[htbp!]
\centering
\footnotesize
\begin{tabular}{cc}
    \includegraphics[trim={30cm 5cm 30.5cm 11cm},clip,scale=0.12]{pics/3D/phi3d.0000.png} 
 & \hspace{-5em}\includegraphics[trim={30cm 5cm 30.5cm 11cm},clip,scale=0.12]{pics/3D/phi3d.0100.png}  \\ 
    \hspace{-0em} \includegraphics[trim={30cm 5cm 30.5cm 11cm},clip,scale=0.12]{pics/3D/phi3d.0160.png} 
    & \hspace{-5em}
    \includegraphics[trim={30cm 5cm 30.5cm 11cm},clip,scale=0.12]{pics/3D/phi3d.0200.png} 
\end{tabular}
\caption{Snapshots of the phase field $\phi$ at the times $t\in\{0,0.5,0.8,1\}$  with mesh size $h_{\max} \approx 2\cdot 10^{-2}$ and time-step $\tau=5\cdot 10^{-4}$. Visualisation of $\phi\geq 0.9$, that is, the tumour phase. \label{pic:3d}}
\end{figure}
\end{comment}

\begin{figure}[htbp!]
\centering
\includegraphics[width=.4\textwidth]{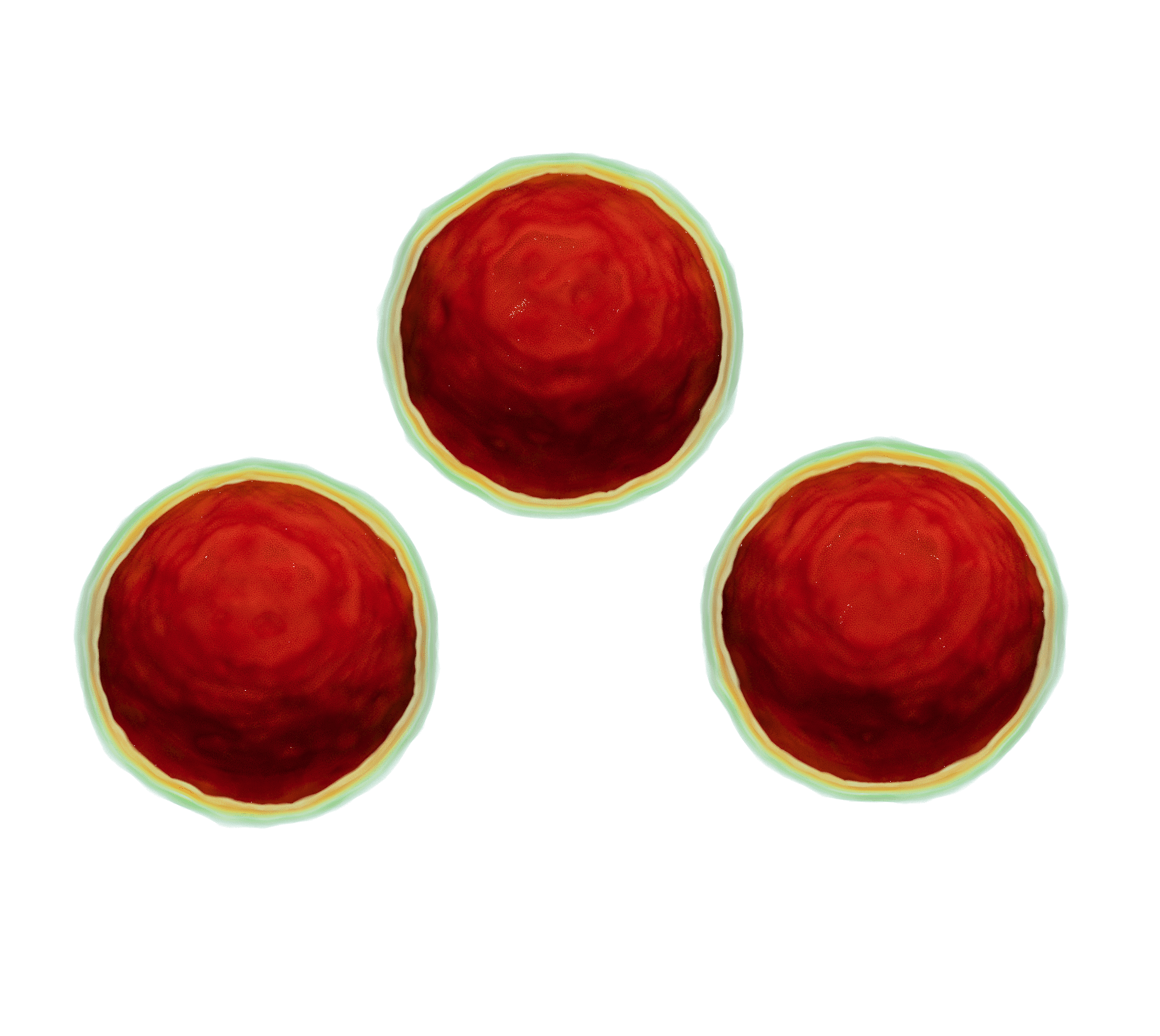} \qquad
\includegraphics[width=.4\textwidth]{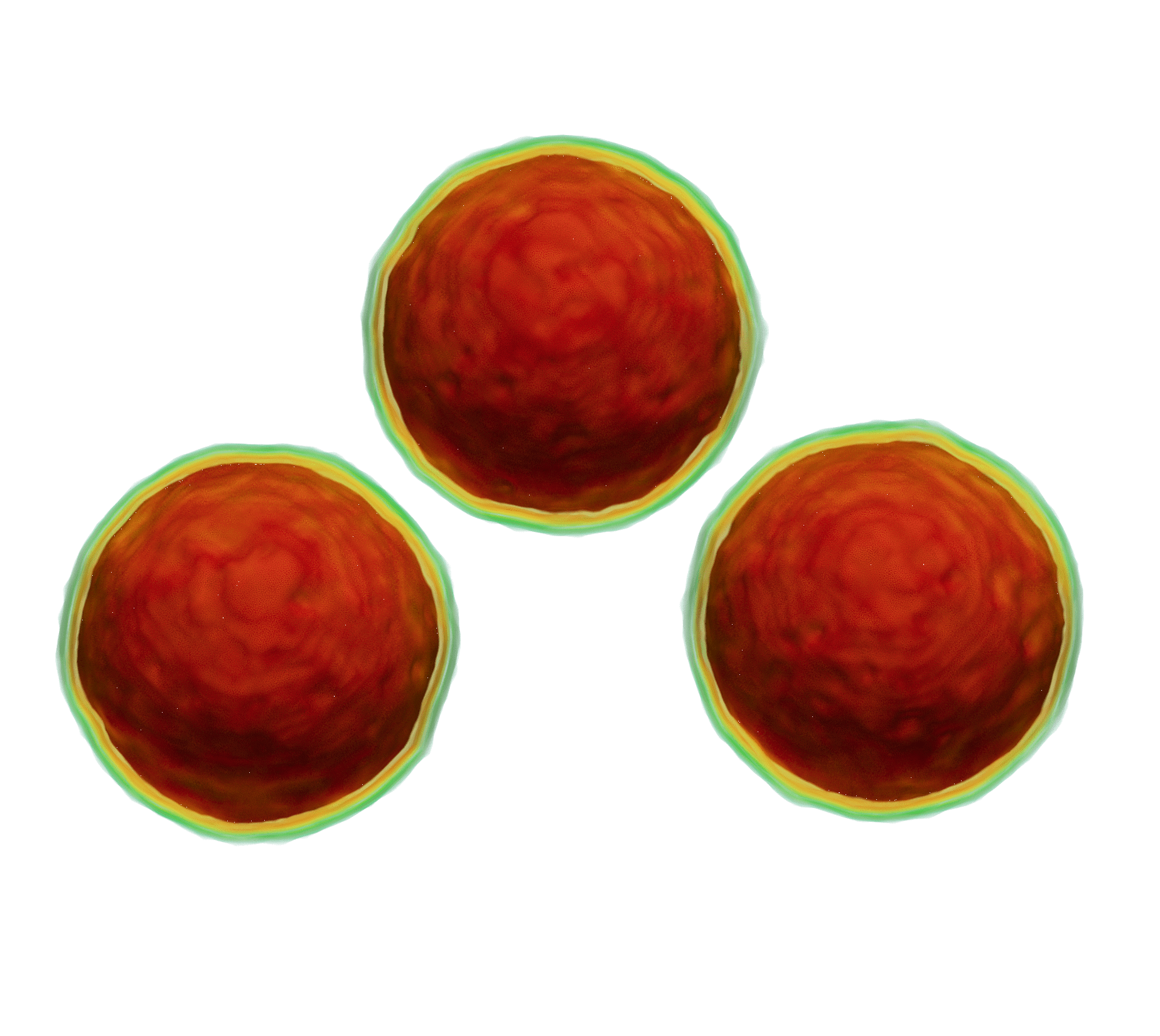} \\
\includegraphics[width=.4\textwidth]{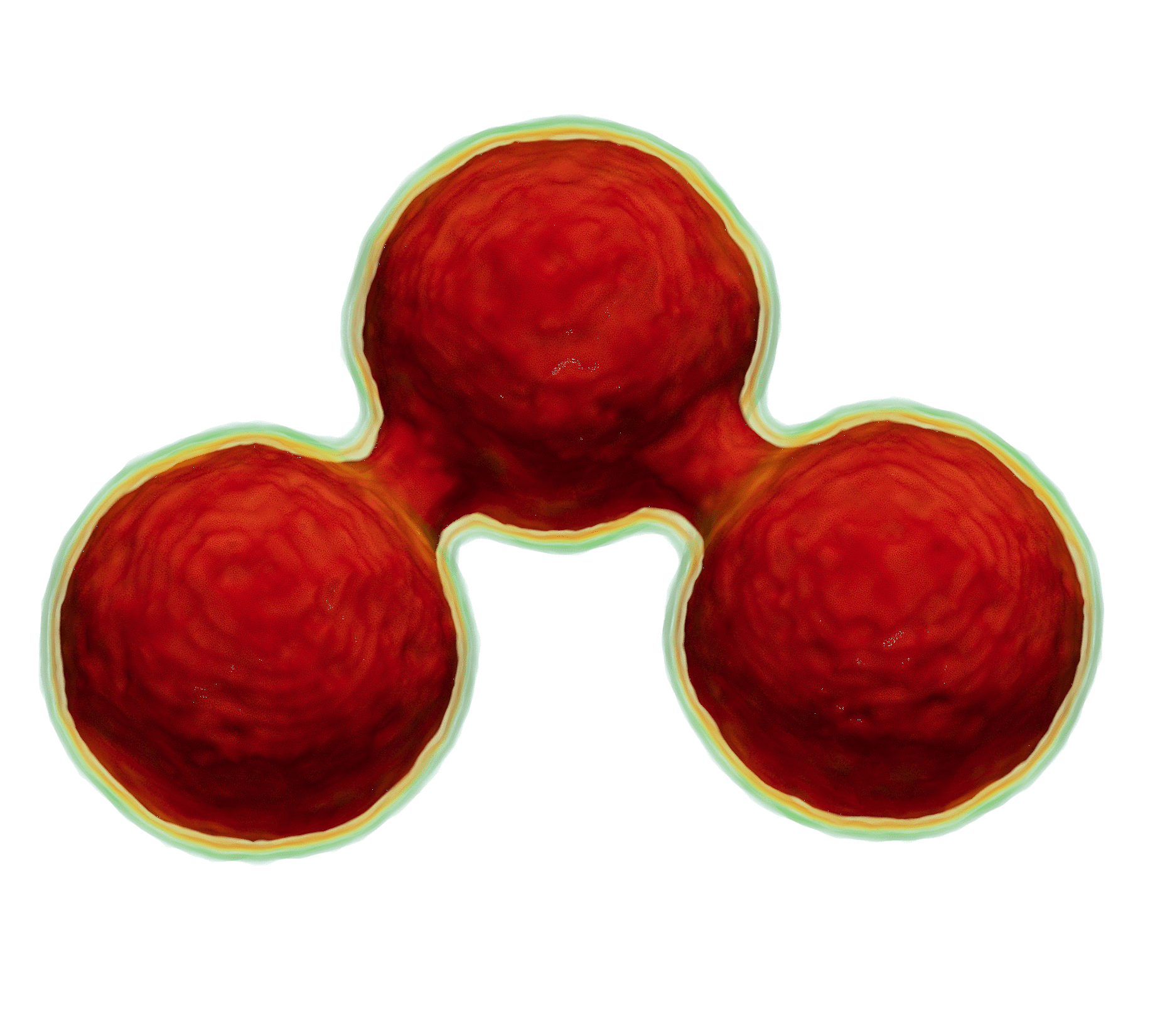} \qquad
\includegraphics[width=.4\textwidth]{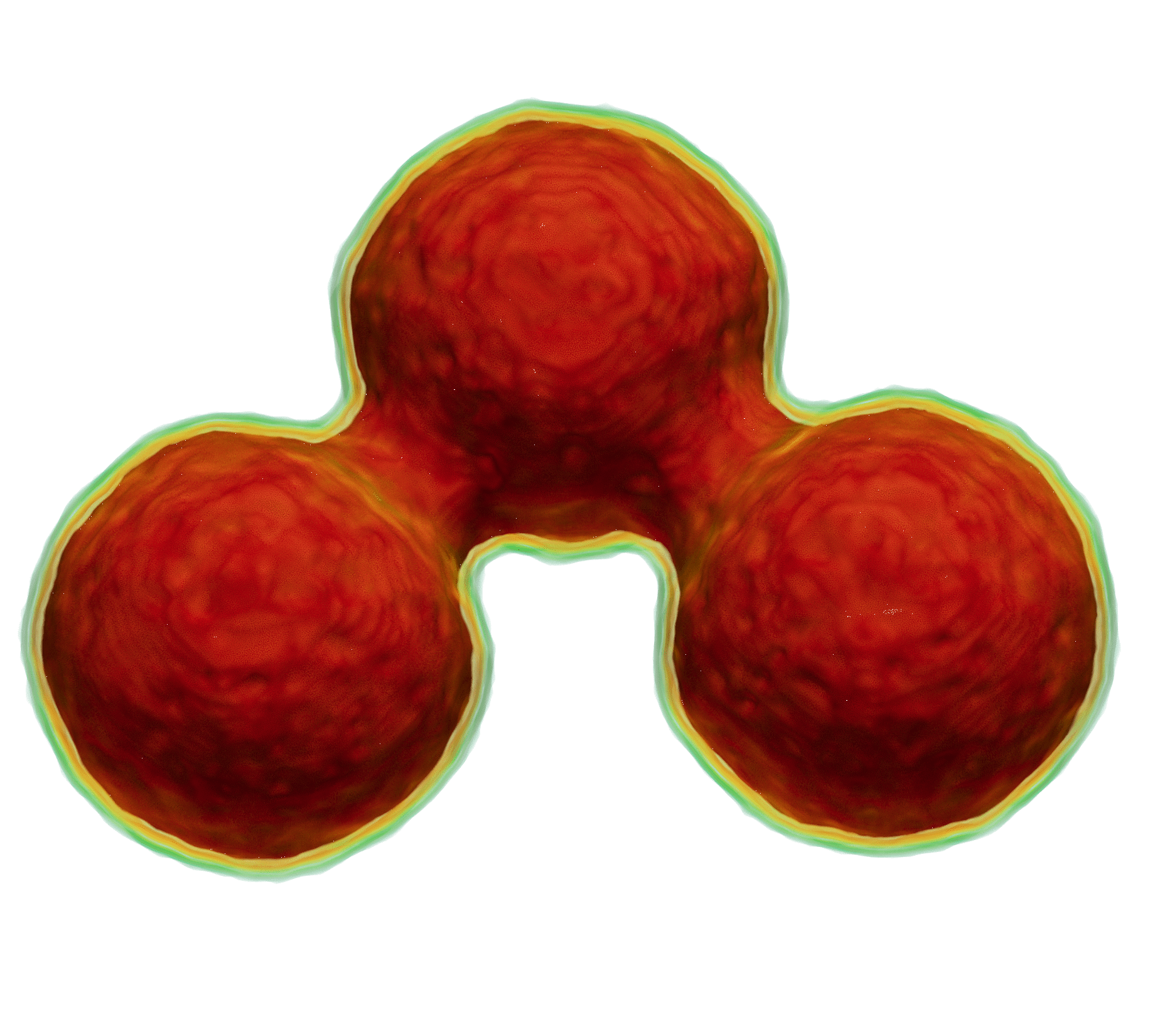}
\caption{Snapshots of the phase field $\phi$ at the times $t\in\{0.2,0.5,0.8,1\}$  with mesh size $h_{\max} \approx 2\cdot 10^{-2}$ and time-step $\tau=5\cdot 10^{-4}$. Visualisation of the tumour core ($\phi\geq 0.9$) in red and the isosurfaces $\phi=-0.25$ (green) and $\phi=0.15$ (yellow) (from top left to bottom right). \label{pic:3d}}
\end{figure}

\section{Conclusion and Outlook} \label{sec:outlook}

In this paper, we proposed a structure-preserving splitting scheme for the Cahn-Hilliard-Biot system using standard finite elements in space and a problem-adapted implicit-explicit Euler time-integration method. We have shown that the method under typical assumptions has discrete solutions, which preserve the thermodynamic structure, that is, the balance of mass, volumetric fluid content, and energy. Furthermore, under more restrictive assumptions, the uniqueness of a discrete solution is guaranteed. The theoretical results are accompanied by numerical tests, i.e. a convergence test. In addition, several test scenarios with application to tumour growth are considered and illustrated.

\section*{Acknowledgement}
A.B.~gratefully acknowledges the support of the German Science Foundation (DFG) via TRR~146 (project~C3) and SPP2256 Project Number 441153493. M.F.~is supported by the State of Upper Austria. 

\def\bibfont{\large}
\bibliography{lit.bib}

\begin{thebibliography}{10}
\providecommand{\url}[1]{\texttt{#1}}
\providecommand{\urlprefix}{URL }
\expandafter\ifx\csname urlstyle\endcsname\relax
  \providecommand{\doi}[1]{doi:\discretionary{}{}{}#1}\else
  \providecommand{\doi}{doi:\discretionary{}{}{}\begingroup \urlstyle{rm}\Url}\fi

\bibitem{abels2024existence}
H.~Abels, H.~Garcke, and J.~Haselböck, \textit{{Existence of Weak Solutions to a Cahn--Hilliard--Biot System}}, arXiv preprint arXiv:2403.07515  (2024). \url{https://doi.org/10.48550/arXiv.2403.07515}.

\bibitem{barrett2006finite}
J.~Barrett, H.~Garcke, and R.~N{\"u}rnberg, \textit{Finite element approximation of a phase field model for surface diffusion of voids in a stressed solid}, Math. Comput. \textbf{75} (2006), no. 253, 7--41. \url{https://doi.org/10.1090/S0025-5718-05-01802-8}.

\bibitem{Barrett1999}
J.~W. Barrett, J.~F. Blowey, and H.~Garcke, \textit{Finite element approximation of the {C}ahn--{H}illiard equation with degenerate mobility}, SIAM J. Numer. Anal. \textbf{37} (1999), no.~1, 286--318. \urlprefix\url{https://doi.org/10.1137/S0036142997331669}.

\bibitem{bartels2015numerical}
S.~Bartels, \textit{Numerical Methods for Nonlinear Partial Differential Equations}, Springer, 2015. \url{https://doi.org/10.1007/978-3-319-13797-1}.

\bibitem{bartels2010posteriori}
S.~Bartels and R.~M{\"u}ller, \textit{{A posteriori error controlled local resolution of evolving interfaces for generalized Cahn--Hilliard equations}}, Interface Free. Bound. \textbf{12} (2010), no.~1, 45--74. \url{https://doi.org/10.4171/ifb/226}.

\bibitem{bendimerad2022structure}
A.~Bendimerad-Hohl, G.~Haine, D.~Matignon, and B.~Maschke, \textit{{Structure-preserving discretization of a coupled Allen--Cahn and heat equation system}}, IFAC-PapersOnLine \textbf{55} (2022), no.~18, 99--104. \url{https://doi.org/10.1016/j.ifacol.2022.08.037}.

\bibitem{blesgen2013cahn}
T.~Blesgen and I.~V. Chenchiah, \textit{{Cahn--Hilliard equations incorporating elasticity: Analysis and comparison to experiments}}, Philos. Trans. A. Math. Phys. Eng. Sci. \textbf{371} (2013), no. 2005, 20120342. \url{https://doi.org/10.1098/rsta.2012.0342}.

\bibitem{bottcher2020structure}
A.~B{\"o}ttcher and H.~Egger, \textit{Structure preserving discretization of {A}llen--{C}ahn type problems modeling the motion of phase boundaries}, Vietnam J. Math. \textbf{48} (2020), 847--863. \url{https://doi.org/10.1007/s10013-020-00428-w}.

\bibitem{BrennerScott}
S.~C. Brenner and L.~R. Scott, \textit{{The Mathematical Theory of Finite Element Methods}}, Springer, 2008. \url{https://doi.org/10.1007/978-0-387-75934-0}.

\bibitem{brunk2023second}
A.~Brunk, H.~Egger, and O.~Habrich, \textit{{A second-order structure-preserving discretization for the Cahn--Hilliard/Allen--Cahn system with cross-kinetic coupling}}, arXiv preprint arXiv:2308.01638  (2023). \url{https://doi.org/10.48550/arXiv.2308.01638}.

\bibitem{brunk2024structure}
A.~Brunk and D.~Schumann, \textit{{Structure-preserving approximation for the non-isothermal Cahn--Hilliard--Navier--Stokes system}}, arXiv preprint arXiv:2402.00147  (2024). \url{https://doi.org/10.48550/arXiv.2402.00147}.

\bibitem{brunk2022second}
A.~Brunk, H.~Egger, O.~Habrich, and M.~Luk\'{a}\v{c}ov\'{a}-Medvi{\softd}ov\'{a}, \textit{{A second-order fully-balanced structure-preserving variational discretization scheme for the Cahn--Hilliard--Navier--Stokes system}}, Math. Models Methods Appl. Sci. \textbf{33} (2023), no.~12, 2587--2627. \url{https://doi.org/10.1142/S0218202523500562}.

\bibitem{BrunkCh}
A.~Brunk, H.~Egger, O.~Habrich, and M.~Luk\'{a}\v{c}ov\'{a}-Medvi{\softd}ov\'{a}, \textit{{Stability and discretization error analysis for the Cahn--Hilliard system via relative energy estimates}}, ESAIM Math. Model Numer. Anal. \textbf{57} (2023), no.~3, 1297--1322. \url{https://doi.org/10.1051/m2an/2023017}.

\bibitem{brunk2023variational}
A.~Brunk, O.~Habrich, T.~D. Oyedeji, Y.~Yang, and B.-X. Xu, \textit{Variational approximation for a non-isothermal coupled phase-field system: {S}tructure-preservation \& nonlinear stability}, arXiv preprint arXiv:2312.14566  (2023). \url{https://doi.org/10.48550/arXiv.2312.14566}.

\bibitem{tumorstresscheng}
G.~Cheng, J.~Tse, R.~Jain, and L.~Munn, \textit{Micro-environmental mechanical stress controls tumor spheroid size and morphology by suppressing proliferation and inducing apoptosis in cancer cells}, PLoS One \textbf{4} (2009), no.~2, e4632. \url{https://doi.org/10.1371/journal.pone.0004632}.

\bibitem{dede2012isogeometric}
L.~Ded{\`e}, M.~J. Borden, and T.~J. Hughes, \textit{Isogeometric analysis for topology optimization with a phase field model}, Arch. Comput. Methods Eng. \textbf{19} (2012), 427--465. \url{https://doi.org/10.1007/s11831-012-9075-z}.

\bibitem{egger2019structure}
H.~Egger, \textit{Structure preserving approximation of dissipative evolution problems}, Numer. Math. \textbf{143} (2019), 85--106. \url{https://doi.org/10.1007/s00211-019-01050-w}.

\bibitem{egger2021structure}
H.~Egger and M.~Sabouri, \textit{On the structure preserving high-order approximation of quasistatic poroelasticity}, Math. Comput. Simul. \textbf{189} (2021), 237--252. \url{https://doi.org/10.1016/j.matcom.2020.12.029}.

\bibitem{eyre1998}
D.~J. Eyre, \textit{{Unconditionally gradient stable time marching the Cahn-Hilliard equation}}, Mater. Res. Soc. Symp. Proc. \textbf{529} (1998), 39--46. \url{https://doi.org/10.1557/PROC-529-39}.

\bibitem{fritz2023wellposedness}
M.~Fritz, \textit{{On the well-posedness of the Cahn--Hilliard--Biot model and its applications to tumor growth}}, arXiv preprint arXiv:2310.07050  (2023). \url{https://doi.org/10.48550/arXiv.2310.07050}.

\bibitem{fritz2023tumor}
M.~Fritz, \textit{Tumor evolution models of phase-field type with nonlocal effects and angiogenesis}, Bull. Math. Biol. \textbf{85} (2023), no.~6, 44. \url{https://doi.org/10.1007/s11538-023-01151-6}.

\bibitem{garcke2022viscoelastic}
H.~Garcke, B.~Kov{\'a}cs, and D.~Trautwein, \textit{{Viscoelastic Cahn--Hilliard models for tumor growth}}, Math. Models Methods Appl. Sci. \textbf{32} (2022), no.~13, 2673--2758. \url{https://doi.org/10.1142/S0218202522500634}.

\bibitem{garcke2021phase}
H.~Garcke, K.~F. Lam, and A.~Signori, \textit{{On a phase field model of Cahn--Hilliard type for tumour growth with mechanical effects}}, Nonlinear Anal. Real World Appl. \textbf{57} (2021), 103192. \url{https://doi.org/10.1016/j.nonrwa.2020.103192}.

\bibitem{garcke2007stress}
H.~Garcke, R.~N{\"u}rnberg, and V.~Styles, \textit{Stress-and diffusion-induced interface motion: {M}odelling and numerical simulations}, Eur. J. Appl. Math. \textbf{18} (2007), no.~6, 631--657. \url{https://doi.org/10.1017/S095679250700719X}.

\bibitem{GarckeTrautwein2022}
H.~Garcke and D.~Trautwein, \textit{Numerical analysis for a {C}ahn–{H}illiard system modelling tumour growth with chemotaxis and active transport}, J. Numer. Math. \textbf{30} (2022), no.~4, 295--324. \urlprefix\url{https://doi.org/10.1515/jnma-2021-0094}.

\bibitem{garcke2005CHENumerics}
H.~Garcke and U.~Weikard, \textit{{Numerical approximation of the Cahn--Larché equation}}, Numer. Math. \textbf{100} (2005), no.~4, 639–662. \url{https://doi.org/10.1007/s00211-004-0578-x}.

\bibitem{graser2014numerical}
C.~Gr{\"a}ser, R.~Kornhuber, and U.~Sack, \textit{Numerical simulation of coarsening in binary solder alloys}, Comput. Mater. Scie. \textbf{93} (2014), 221--233. \url{https://doi.org/10.1016/j.commatsci.2014.06.010}.

\bibitem{tumorstresshelminger}
G.~Helmlinger, P.~A. Netti, H.~C. Lichtenbeld, R.~J. Melder, and R.~K. Jain, \textit{Solid stress inhibits the growth of multicellular tumor spheroids}, Nat. Biotechnol. \textbf{15} (1997), no.~8, 778--783. \url{https://doi.org/10.1038/nbt0897-778}.

\bibitem{lan2023operator}
R.~Lan, J.~Li, Y.~Cai, and L.~Ju, \textit{Operator splitting based structure-preserving numerical schemes for the mass-conserving convective allen-cahn equation}, J. Comput. Phys. \textbf{472} (2023), 111695. \url{https://doi.org/10.1016/j.jcp.2022.111695}.

\bibitem{lima2016}
E.~Lima, J.~T. Oden, D.~Hormuth, T.~Yankeelov, and R.~Almeida, \textit{Selection, calibration, and validation of models of tumor growth}, Math. Models Methods Appl. Sci. \textbf{26} (2016), no.~12, 2341--2368. \url{https://doi.org/10.1142/S021820251650055X}.

\bibitem{lima2017}
E.~Lima et~al., \textit{Selection and validation of predictive models of radiation effects on tumor growth based on noninvasive imaging data}, Comput. Methods Appl. Mech. Eng. \textbf{327} (2017), 277--305. \url{https://doi.org/10.1016/j.cma.2017.08.009}.

\bibitem{milosevictumor}
M.~Milosevic, S.~Lunt, E.~Leung, J.~Skliarenko, P.~Shaw, A.~Fyles, and R.~Hill, \textit{Interstitial permeability and elasticity in human cervix cancer}, Microvasc. Res. \textbf{75} (2008), no.~3, 381--390. \url{https://doi.org/10.1016/j.mvr.2007.11.003}.

\bibitem{riethmuller2023wellposedness}
C.~Riethm{\"u}ller, E.~Storvik, J.~W. Both, and F.~A. Radu, \textit{{Well-posedness analysis of the Cahn--Hilliard--Biot model}}, arXiv preprint arXiv:2310.18231  (2023). \url{https://doi.org/10.48550/arXiv.2310.18231}.

\bibitem{schoberl2014c++}
J.~Sch{\"o}berl, \textit{{C$++$11 implementation of finite elements in NGSolve}}, ASC (Institute for Analysis and Scientific Computing) Report, TU Wien \textbf{30} (2014), 1--23. \url{http://hdl.handle.net/20.500.12708/28346}.

\bibitem{shimura2020error}
K.~Shimura and S.~Yoshikawa, \textit{{Error estimate for structure-preserving finite difference schemes of the one-dimensional Cahn--Hilliard system coupled with viscoelasticity}}, \textit{Regularity and Asymptotic Analysis for Critical Cases of Partial Differential Equations}, Research Institute for Mathematical Sciences, Kyoto University, 2020. 159--175. \url{http://hdl.handle.net/2433/260682}.

\bibitem{shimura2022new}
K.~Shimura and S.~Yoshikawa, \textit{{A new conservative finite difference scheme for 1D Cahn--Hilliard equation coupled with elasticity}}, J. Appl. Anal. \textbf{28} (2022), no.~2, 311--332. \url{https://doi.org/10.1515/jaa-2021-2071}.

\bibitem{STORVIK2022}
E.~Storvik, J.~W. Both, J.~M. Nordbotten, and F.~A. Radu, \textit{{A Cahn--Hilliard--Biot system and its generalized gradient flow structure}}, Appl. Math. Lett. \textbf{126} (2022), 107799. \url{https://doi.org/10.1016/j.aml.2021.107799}.

\bibitem{storvik2024sequential}
E.~Storvik, C.~Riethm{\"u}ller, J.~W. Both, and F.~A. Radu, \textit{{Sequential solution strategies for the Cahn--Hilliard--Biot model}}, arXiv preprint arXiv:2401.13358  (2024). \url{https://doi.org/10.48550/arXiv.2401.13358}.

\bibitem{tumorstressstylianopoulos}
T.~Stylianopoulos et~al., \textit{Causes, consequences, and remedies for growth-induced solid stress in murine and human tumors}, Proc. Natl. Acad. Sci. \textbf{109} (2012), no.~38, 15101--15108. \url{https://doi.org/10.1073/pnas.1213353109}.

\bibitem{Wloka}
J.~Wloka, \textit{Partial Differential Equations}, Cambridge University Press, 1987. \url{https://doi.org/10.1017/CBO9781139171755}.

\bibitem{Zeidler1}
E.~Zeidler, \textit{Nonlinear Functional Analysis and its Applications~{I}: Fixed-Point Theorems}, Springer, 1986.

\end{thebibliography}

\end{document}